\RequirePackage[table]{xcolor}                                                                                                                                                                                                                                                                                                                                                                                                                                                                                                                                                                                                                                                                                                                                                                                                                                                                                                                                                                                                                                                                                                                                                                                                                                                                                                                                                                                                                                                                                                                                                                                                                                                                                                                                                                                                                                                                                                                                                                                                                                                                                                                                                                                                                                                                                                                                                                                                                                                                                                                                                                                                                                                                                                                                                                                                                                                                                                                           
\documentclass[a4paper,12pt]{amsart}
\usepackage[T1]{fontenc}
\usepackage{CJK}
\usepackage{romannum}
\usepackage{amsfonts}
\usepackage{mathtools}
\usepackage{amsmath,amscd}

\usepackage{ifthen}
\usepackage{amsrefs}
\usepackage{mathrsfs}
\usepackage{amsthm}
\usepackage{amssymb}

\usepackage{tikz-cd}
\usepackage{graphicx}
\usepackage{relsize}

\usepackage{MnSymbol}

\usepackage{hyperref}
\usepackage{url}
\usepackage{etoolbox}
\usepackage{rotating}

\usepackage[shortlabels]{enumitem}
\usepackage[paper=a4paper,left=20mm,right=20mm,top=25mm,bottom=30mm]{geometry}

\usepackage{tikz}
\usepackage{mleftright}
\usetikzlibrary{backgrounds}
\usepackage{tkz-euclide}
\usepackage{xcolor}
\usetikzlibrary{trees,snakes,shapes.geometric}
\usepackage[outline]{contour}
\contourlength{1.5pt}
\usetikzlibrary{positioning}
\usetikzlibrary{%
  matrix,%
  calc,%
  arrows%
}

\setlist[enumerate]{topsep=0em, itemsep= -0em, parsep = 0 em, label=$(\alph*)$}
\let\emptyset\varnothing
\nocite{*}

\newcommand{\cA}{\mathcal{A}}

\newcommand{\cC}{\mathcal{C}}
\newcommand{\cD}{\mathcal{D}}

\newcommand{\bg}{\bar{\gamma}}

\newcommand{\cK}{\mathcal{K}}

\newcommand{\cN}{\mathcal{N}}

\newcommand{\cT}{\mathcal{T}}

\newcommand{\cW}{\mathcal{W}}

\newcommand{\cS}{\mathcal{S}}

\DeclareMathOperator{\rep}{rep}
\DeclareMathOperator{\Rep}{Rep}

\DeclareMathOperator{\rk}{rk}

\DeclareMathOperator{\modd}{mod}
\DeclareMathOperator{\supp}{supp}

\DeclareMathOperator{\im}{im}

\DeclareMathOperator{\EIP}{EIP}

\DeclareMathOperator{\Inj}{Inj}
\DeclareMathOperator{\Sur}{Surj}

\DeclareMathOperator{\EKP}{EKP}

\DeclareMathOperator{\GR}{GR}
\DeclareMathOperator{\CR}{CR}

\let\emptyset\varnothing
\newtheorem{thm}{Theorem}[subsection]
\newtheorem{cor}[thm]{Corollary}

\newtheorem{proposition}{Proposition}[section]
\newtheorem{Theorem}[proposition]{Theorem}
\newtheorem{Lemma}[proposition]{Lemma}

\newtheorem{corollary}[proposition]{Corollary}

\newtheorem{TheoremA}{Theorem}
\newtheorem{example}{Example}

\newenvironment{Definition}[1][Definition.]{\begin{trivlist}
\item[\hskip \labelsep {\bfseries #1}]}{\end{trivlist}}

\title{Graded Kronecker Modules of Constant Rank }
\author{Jie Liu}

\address{School of Mathematics and Statistics, Guangdong University of Technology
Guangzhou 510520, Guangdong, People's Republic of China }

\email{jie@gdut.edu.cn}

\address{ShenZhen International Center For Mathematics,   Southern University of Science and Technology, shenzhen 518055, China }

\begin{document}

%%%%%%%%%%%%%%%%%%%%%% TITELSEITE %%%%%%%%%%%%%%%%%%%%%%%%%%%%%%%%
 
\rmfamily

\maketitle

%\pagenumbering{roman}

\begin{abstract}

Let $K(n)$ be the generalized Kronecker quiver, and let  $(T(n),\Omega)$ be the universal covering of $K(n)$, where $n\geq 3$ and $\Omega$ is a bipartite orientation. We introduce three different types of representations for quiver $(T(n),\Omega)$ by $K(n)$: maps are injective in a representation maps are all surjective in a representation, and representations are of constant rank. We use $\Inj, \Sur$ and $\GR$ to denote the corresponding subcategories containing these representations, respectively, and we show that

\begin{center}
 $\GR= \Inj\cup\Sur$.
\end{center}

\end{abstract}
\providecommand{\keywords}[1]
{
  \small	
  \textbf{\text{Keywords:}} #1
}


\pagenumbering{arabic}

\section{Introduction}

Let $k$ be an algebraically closed field. Let $K(n)$ be the generalized Kronecker quiver:  $ 
\begin{tikzcd}
  \circ
    \arrow[r, draw=none, "{\vdots}" description]
    \arrow[r, bend left,        "\gamma_1"]
    \arrow[r, bend right, swap, "\gamma_n"]
    &
    \circ ,
\end{tikzcd}
 $  and let $\cK_n$ be the path algebra of $K(n)$. We use $\modd \cK_n$ to denote the category of the finite-dimensional modules of $\cK_n$. When $n\leq 2$, it is already known how to classify  the indecomposable modules in $\modd \cK_n$ \cite[\Romannum{19}.3]{Assem2}. When $n\geq 3$,  however, it is   hopleless to classify all the indecomposable modules since the representation type is wild now (cf.\cite[\Romannum{18}]{Assem2}, \cite[1.3]{Kerner}). Hence it is desirable to find invariants in the category $\modd \cK_n$.   Julia Pevtsova  once introduced the category of modules of constant Jordan type for a finite group scheme $G$ in 2008 \cite[Section 1]{Carlson}. Julia Worch narrowed the topic down to elementary abelian $p$-groups $E_r$ and the generalized Beilinson algebras $B(n,r)$ \cite{Julia}. In particular, she discussed the Beilinson algebra $B(2,r)$,  and the later is isomorphic to $\cK_r$. She reduced the modules of constant Jordan type to the modules of constant rank in $\modd \cK_n$ \cite[section 3]{Julia}. It says that if $M\in \modd \cK_n$, then we have a linear operator $x^{\alpha}_{M}=\alpha_1M(\gamma_1)+\alpha_2M(\gamma_2)+\cdots+\alpha_nM(\gamma_n): M_1\rightarrow M_2$ for every  $\alpha=(\alpha_1,\cdots,\alpha_n)\in k^n\setminus\{0\}$, and   $M$ is said to be of constant rank when $\rk x^{\alpha}_M$ is fixed for any non-zero  $\alpha$.   We let $\CR=\{M\in\modd\cK_n\mid M \text{ is of constant rank}\}$. Furthermore,  we  introduce two full  subcategories of $\modd\cK_n$ with the equal kernels property and with the equal images property : $\EKP:=\{M\in \modd\cK_n\mid \forall \alpha\in k^n\setminus\{0\}: x^{\alpha}_{M} \text{ is injective} \}$,  $\EIP:=\{M\in \modd\cK_n\mid \forall \alpha\in k^n\setminus\{0\}:x^{\alpha}_{M} \text{ is surjective} \}$ (cf.\cite[Definition 2.1]{Julia}). Apparently, we have $\EKP\cup \EIP \subseteq \CR$.

  Let $\tau$ be the Auslander--Reiten translation of category $\modd\cK_n$. We call an indecomposable module $N\in \modd\cK_n$ \textit{regular} if $\tau^t N\neq 0$ for any $t\in \mathbf{Z}$. It is well-known that the regular component of $\modd\cK_n$ is  of type $\mathbb{Z}A_{\infty}$. If $\cD$ is regular component of $\modd\cK_n$, then for any  regular indecomposable module $M\in \EKP\cap \cD$ there exists a number $m\in \mathbb{N}$ such that $\tau^{m+1}M\in \EIP\cap \cD$, and we call such $m$ the \textit{width} of the regular component $\cD$ \cite[Theorem 3.3]{Julia}.  Let $M\in \modd\cK_n$ be a regular indecomposable module. Daniel Bissinger proved that there existed at most one number $r\in \mathbf{Z}$ such that $\tau^r M\in \CR\setminus (\EKP\cup \EIP)$ \cite[Proposition 3.14]{Daniel}.

 Let  $(T(n),\Omega)$ be the universal covering of $K(n)$, i.e., $T(n)$ is  an $n$-regular tree  with a bipartite orientation $\Omega$ \cite[Section 7]{Daniel}. Let $\modd(T(n),\Omega)$ be the category of finite-dimensional representations of $(T(n),\Omega)$. Similarly, we have subcategories $\Inj$, $\Sur$ and $\GR$ in $\modd(T(n),\Omega)$ corresponding to $\EIP$, $\EKP$ and $\CR$ in $\modd\cK_n$. 
 The category $\modd(T(n),\Omega)$ also admits AR--sequences (cf.\cite[2.2]{Bongartz}), and  there exists a push-down functor $\pi_\lambda:\rep_k(T(n),\Omega)\rightarrow\rep_k(K(n))$.  The functor $\pi_\lambda$ is exact (cf. \cite{Daniel}).  An indecomposable representation $M\in \modd(T(n),\Omega)$ is said to be of \textit{constant rank} if the  representation $\pi_\lambda(M)$ is of constant rank. Claus Michael Ringel defined three types of representations: \textit{sink module}, \textit{flow module}, and \textit{source module} in $\modd(T(n),\Omega)$ (cf. \cite{Claus}). In this paper, we investigate  representations of constant rank of $(T(n),\Omega)$  using  Claus's conceptions, and we show that 
 
 \begin{TheoremA}
 
For the category  $\modd(T(n),\Omega)$  we have 
\begin{center}
 $\GR = \Inj\cup\Sur$.
\end{center}

\end{TheoremA}

\section{Preliminaries}
In this section, we present a few concepts and some basic results.  For  convenience, we will give some definitions in a short way. A thorough introduction to this part can be found in  \cite[\Romannum{2}-\Romannum{7}]{Assem1}, \cite{Claus} or \cite{Daniel}. Throughout, $k$ will  denote an algebraically closed field.

\subsection{Graph}  A \textit{graph} is a pair $G=(G_0, G_1):$   $G_0$ (whose elements called the vertices) and $G_1$ (whose elements called the edges). If $\{a, b\}$ is an edge, then $a,b$ are also called neighbours.

For a graph $G$, there exists a map $\Omega$, called \textit{orientation}  $:G \rightarrow G_0\times G_0$, such that $\Omega(\{a, a'\})$ is either ($a, a'$) or ($a', a$). We call $(G, \Omega)$  an \textit{oriented graph}. We  write $a\rightarrow a'$ if $\Omega(\{a, a'\})=(a, a')$ and call $a$ the \textit{start} and $a'$ the \textit{target}. Meanwhile we call $a$   a \textit{sink} (or a \textit{source}) if it is not a start (or the target, respectively) of any arrow. If any vertex is a sink or a source, then the orientation  $\Omega$ will be called \textit{bipartite}. A \textit{subgraph}    of a  graph $G=(G_0,G_1)$ is a graph  $G'=(G'_0,G'_1)$ such that $G'_0\subseteq G_0, G'_1\subseteq G_1$.

A  \textit{walk} of length $t\geq 0$ with \textit{start} $a_0$ and \textit{target} $a_t$ in a graph $G$ is a finite sequence 
\begin{center}
$l_t=(a_0|\{a_0,a_1\},\{a_1,a_2\},\cdots,\{a_{t-1},a_t\}|a_t)$
\end{center}
such that   $a_{i-1}$ and $a_i$ are neighbours for $1\leq i\leq t$, and $a_{j-1}\neq a_{j+1}$ for $1\leq j < t$. We sometimes use $(a_0,a_1,\cdots, a_t)$ to denote the walk $l_t$ if there only exists one walk connecting $a_0$ and $a_t$. When $t=2r$,   $a_r$ is called its \textit{center} and $r$ its \textit{radius} \cite[2.1]{Claus}. When $t=2r+1$,  $\{a_r, a_{r+1}\}$ is called its \textit{center} and $r$ its \textit{radius}. If for any pair $a, b$ of vertices in $G$ there always exists a walk connecting $a$ and $b$, then we say that $G$ is \textit{connected}. We call a walk of length $t \geq 0$ being a \textit{cycle} whenever its source and target coincide. A cycle of length $1$ is said to be a \textit{loop}. A graph is called \textit{acyclic} if it contains on cycles,  and it is called \textit{finite} if $G_0$ and $G_1$ are finite set.

We call a graph $G$  a \textit{tree} if it is connected and walks connecting a vertex with itself are always length 0. Suppose that $G$ is a tree.  We define

\begin{center}
$d(a, b):=$ the length of walk connecting $a$ and $b$.
\end{center} 
This is well-defined since there  only exists one walk from $a$ to $b$ in the tree $G$. If $G'$ is a subset of tree $G$ and $x\in G'$, we define

\begin{center}
$d(x, G'): =$min$\{d(x, a)\mid a \in G'\}$.
\end{center}

For every finite tree $G$, the walks of maximal possible length $d(G)$ are  called the \textit{diameter} of $G$.  Let $r(G)$ denote the radius of $G$. It can be shown that for a finite tree $G$, all diameter walks of $G$ have the same center  \cite[2.1]{Claus}. Suppose that $(a_0,a_1,\cdots,a_t)$ is a diameter walk of $G$.  We  define, 
\begin{center}
\[ c(G)=\begin{cases} 
     a_r & t=2r , \\
     \{a_r,a_{r+1}\} & t=2r+1.
   \end{cases}
\]
\end{center}

\subsection{Quivers}

A \textit{quiver} is just an oriented graph (loops and multiple arrows are allowed), usually denoted by  $Q=(Q_0, Q_1, s, t)$, where $s,t: Q_1\rightarrow Q_0$.  For every $x\in Q_0$, we define a set 

\begin{center}
$\mathcal{N}(x):=\{y\in Q_0\mid \exists \alpha \in Q_1: t(\alpha)=y \text{ when } s(\alpha)=x \text{ or } s(\alpha)=y \text{ when } t(\alpha)=x\}$
\end{center}
and call it the \textit{neighbourhood} of $x$, elements in it are called \textit{neighbour} of $x$, respectively. When $\mid \mathcal{N}(x)\mid \leq 1$,  we say that $x$  is a  \textit{leaf}. A quiver $Q$  is said to be \textit{locally bounded} if for every $x\in Q_0$ the neighbourhood $\cN(x)$ is finite and there exists a number $m_x\in \mathbb{N}$ such that  each walk in $Q$ which starts or ends in $x$ is of length $\leq m_x$. In this note, we always assume that it is locally bounded when we talk about a quiver $Q$.

When $Q$ is a tree, we have the following easy conclusion.

 \begin{Lemma}\label{l}
Let $(a_0, a_1, a_2, \cdots, a_l)$ be a diameter walk of  a finite tree $Q$, where $l\in\mathbb{N}_{\geq 2}$. Then the  set $\mathcal{N}'(a_1)=\mathcal{N}(a_1)\setminus \{a_2\}$ consists of  all leaves.
\end{Lemma}

\begin{proof}
Suppose that there exists a vertex $a'\in \mathcal{N}'(a_1) $ being not a leaf. It  means that there  exists at least one vertex $a_1 \neq a''\in \mathcal{N}(a')$ such that the length of walk $(a'', a', a_1, a_2, \cdots, a_m)$ is greater than that of diameter walk $(a_0,a_1, a_2,\cdots,a_l)$ in the tree $Q$, this is a contradiction. 
\end{proof}

 A finite-dimensional representation $M=((M_x)_{x\in Q_0}, (M(\alpha))_{\alpha\in Q_1})$ over $Q$ consists of vector spaces $M_{x}$ and $k$-linear maps $M(\alpha): M_{s(\alpha)} \rightarrow M_{t(\alpha)}$ such that $\dim_{k}M:=\sum _{x\in Q_0}\dim_k M_{x}$ is finite. Normally, we use $P(i)$, $I(i)$ and $S(i)$ to denote the projective, injective and simple representation at the vertex $i$, respectively,  $i\in Q_0$. A \textit{morphism} $f: M\rightarrow N$ between representations is a collection of $k$-linear maps $(f_z)_{z\in Q_0}$ such that for each arrow $\alpha: x\rightarrow y$ there is a commutative diagram 

\begin{center}
$\begin{array}[c]{ccc}
M_{x}&\stackrel{M(\alpha)}{\longrightarrow}&M_y\\
\downarrow\scriptstyle{f_x}&&\downarrow\scriptstyle{f_y}\\
N_x&\stackrel{N(\alpha)}{\longrightarrow}&N_y.
\end{array}$
\end{center}
We thus define a category $\Rep_k(Q)$  of  $k$-linear representations  of $Q$.  We denote by $\rep_k(Q)$ the full  subcategory of $\Rep_k(Q)$ consisting of finite-dimensional representations.

\subsection{Covering Theory}
 We now focus on the generalized Kronecker quiver $K(n)$, $n\geq 3$, 
\[
\begin{tikzcd}
    1 \circ
    \arrow[r, draw=none, "{\vdots}" description]
    \arrow[r, bend left,        "\gamma_1"]
    \arrow[r, bend right, swap, "\gamma_n"]
    &
    \circ 2.
\end{tikzcd}
\]

Let $ \mathcal{K}_n=kK(n)$. Clearly, algebra $\cK_n$ is Artin \cite[Section 2.1]{Julia}.  Let $\epsilon_i$ correspond to the trivial path point $i$, $i\in\{1, 2\}.$  We now give a brief  construction of the covering $T(n)$ of Kronecker quiver $K(n)$ universally \cite[7.1]{Daniel}. Let $Q=(Q_0,Q_1,s,t)$ be a quiver.  We define formal inverse $(Q_1)^{-1}$ on $Q_1$, where

\begin{enumerate}
\item  $(Q_1)^{-1}:=\{\alpha^{-1}\mid \alpha \in Q_1\}$;

\item $s(\alpha^{-1}):=t(\alpha)$ and $t(\alpha^{-1}):=s(\alpha)$ for any $\alpha^{-1}\in (Q_1)^{-1}$;

\item  we say that $\omega$  is a \textit{path} of $Q_1$, provided there exists $\omega=\alpha^{\varepsilon_n}_n\cdots \alpha^{\varepsilon_1}_1$, where $\alpha_i\in Q_1$, $\varepsilon\in\{1, -1\}$ and $s(\alpha^{\varepsilon_{i+1}}_{i+1})=t(\alpha^{\varepsilon_i}_i)$ for all $i<n$;
\item $s(\omega):=s(\alpha^{\varepsilon_1}_1)$ and $t(\omega):=t(\alpha^{\varepsilon_n}_n)$.

\end{enumerate}

Let $W:=\{\text{paths of $K(n)$}\}$. We introduce an equivalence relation $\sim$  on  $W$, which is generated by $\gamma^{-1}_i \gamma_i\sim \epsilon_1$ and $\gamma_i\gamma^{-1}_i\sim \epsilon_2$. We define an involution $(-)^{-1}$ on $W$, where 
\begin{center}
$(-)^{-1}$: $W\rightarrow W; (\alpha^{\varepsilon_n}_n\cdots \alpha^{\varepsilon_1}_1)^{-1} \mapsto \alpha^{-\varepsilon_1}_1\cdots \alpha^{-\varepsilon_n}_n$.
\end{center}
 Let $\pi(K(n))$  be the fundamental group of $K(n)$ at the point 1, i.e. $\pi(K(n))=\{[\alpha]\mid s(\alpha)=t(\alpha)=1\}$, where $[\alpha]$ is the equivalence classe of unoriented path $\alpha$. The multiplication is  the concatenation of paths. We define [$\omega$]$^{-1}$:=[$\omega^{-1}$] and the identity element is [$\epsilon_1$].  Then quiver $T(n)$ is given by the following data:
\begin{enumerate}

\item $(T(n))_0$ is the set of equivalence classes of paths starting at 1.

\item When $\omega^{'}\sim \gamma_i \omega$ for some $i\in\{1, \cdots, r\}$, we say that there exists an arrow from [$\omega$] to [$\omega^{'}$]. 
\end{enumerate}

Let $\pi: T(n)\rightarrow K(n); [\omega] \mapsto t(\omega), ([\omega]\rightarrow [\gamma_i \omega]) \mapsto \gamma_i$.  Action of group $G=\pi(K(n))$ on $T(n)$ is the concatenation of paths: 
\begin{center}
$g.[v]=[v\omega^{-1}]$ and
$g.([u]\rightarrow[\gamma_iu])=([u\omega^{-1}]\rightarrow [\gamma_i u\omega^{-1}])$,
\end{center}
where  $g=[\omega]\in\pi(K(n))$ and $[v], [u]\in T(n)_0$ with arrow $[u]\rightarrow [\gamma_i u]=[v]$.
Finally, we define $T^+_{n}:=\pi^{-1}(\{1\})$, $T^-_n:=\pi^{-1}(\{2\})$.  Let $\Omega$ be the bipartite  on $T(n)$.  Then for any  $M\in$ rep$_k(T(n), \Omega)$,  group $G$ acts on $M$ via:
\begin{center}
$M^g:=(((M^g)_x)_{x\in T(n)_0}$, $(M^g(\alpha))_{\alpha\in T(n)_1})$,
\end{center}
where $(M^g)_x:= M_{g.x}$ and $M^g(\alpha):=M(g.\alpha)$. 

\quad

We now give several definitions.

\begin{Definition} Let $ \emptyset\neq S\subseteq T(n)_0$ be a set of vertices.
\item[$(a)$] The unique minimal tree in $T(n)$ containing $S$ is denoted by $T(S)$.

\item[$(b)$] A vertex $x\in T(S)_0$ is called a \textit{leaf} of $T(S)$,  provided $\mid \mathcal{N}(x)\cap T(S)_0\mid \leq 1$. 

\end{Definition}
\begin{Definition}
 Let $x\in T(n)_0$ and $M\in$ mod$(T(n), \Omega)$ be a module.
\begin{enumerate}
\item[$(a)$] The set supp$(M):=\{y\in T(n)_0\mid M_y\neq0\}$ is called the \textit{support} of $M$.

\item[$(b)$] The vertex $x$ is said to be a \textit{leaf} $M$,  provided $x$ is a leaf of $T(M):=T(\text{supp}(M))$. 

\end{enumerate}

 \end{Definition}

Let $M\in\modd(T(n),\Omega)$. We call $M$ a \textit{thin} module, provided dim$_kM_a\leq 1$ for every $a\in$ $T(n)_0$. Suppose that $M\in$ mod$(T(n),\Omega)$ is a thin module. According to \cite[Proposition 1]{Claus3}, we know that $M$ is indecomposable if supp$(M)$ is connected. Now we look at an example.

\begin{example}
Let $M\in\rep_k(T(3), \Omega)$, and let $T(M)$ be the following:

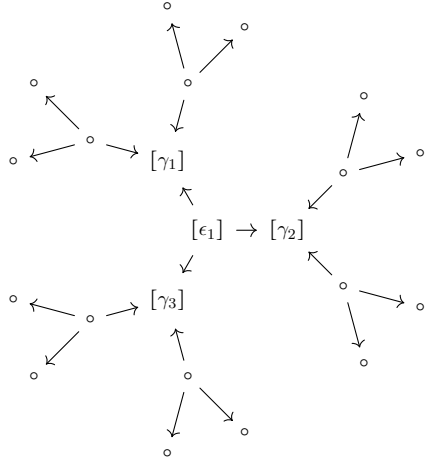
\tikzstyle{level 1}=[sibling angle=120,level distance = 30, ->]
\tikzstyle{level 2}=[sibling angle=90,level distance = 30, <-]
\tikzstyle{level 3}=[sibling angle=60,level distance = 30, ->]
\tikzstyle{level 4}=[sibling angle=45,level distance = 30, <-]
\tikzstyle{every node}=[]
\tikzstyle{edge from parent}=[segment angle=10,draw]
\begin{center}

\begin{tikzpicture}[grow cyclic, shape=circle,cap=round, scale=1, every node/.style={scale=0.7}]
\node {$[\epsilon_1]$} 
    child {node{$[\gamma_3]$} child {node {$\circ$} child {node{$\circ$}} child {node {$\circ$}}} child {node {$\circ$} child {node{$\circ$}} child {node{ $\circ$}}}}
    child { node {$[\gamma_2]$} child {node {$\circ$} child {node {$\circ$}} child {node {$\circ$}}} child {node {$\circ$} child {node {$\circ$}} child {node{$\circ$}}}}
    child { node {$[\gamma_1]$} child {node {$\circ$} child {node {$\circ$}} child {node {$\circ$}}} child {node {$\circ$} child {node {$\circ$}} child {node{$\circ$}}}};

\end{tikzpicture}
\end{center}
 where $\dim_kM_a=1$, $\dim_k M_b=0$ for all $a\in S$, $b\nin S$,  and $M(\gamma_i)=\lambda_i$, $\lambda_i\in  k\setminus\{0\}$,  where  $S=\{ [\epsilon_i], [\gamma_i]\mid 1\leq i\leq 3 \}$. Let $g=[\gamma^{-1}_2\gamma_1]\in G$. Then $M$ is indecomposable and supp($M^g$) is the following  \cite[Proposition 1]{Claus5}:

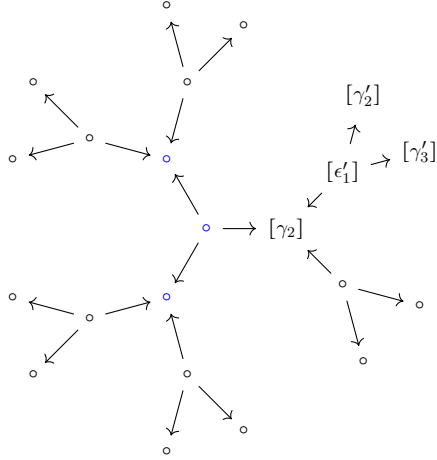
\begin{center}

\begin{tikzpicture}[grow cyclic, shape=circle,cap=round, scale=1, every node/.style={scale=0.7}]
\node {\textcolor{blue}{$\circ$}} 
    child { node {\textcolor{blue}{$\circ$}} child {node {$\circ$} child {node {$\circ$}} child {node {$\circ$}}} child {node {$\circ$} child {node {$\circ$}} child {node {$\circ$}}}}
    child {node{$[\gamma_2]$} child {node {$\circ$} child {node {$\circ$}} child {node {$\circ$}}} child {node {$[\epsilon'_1]$} child {node {$[\gamma'_3]$}} child {node {$[\gamma'_2]$}}}}
    child {node{\textcolor{blue}{$\circ$}} child {node {$\circ$} child {node {$\circ$}} child {node {$\circ$}}} child {node {$\circ$} child {node {$\circ$}} child {node {$\circ$}}}};
\end{tikzpicture}
\end{center}
where $[\epsilon'_1]=[\gamma^{-1}_1\gamma_2], [\gamma'_2]=[\gamma_2\gamma^{-1}_1\gamma_2],[\gamma'_3]=[\gamma_3\gamma^{-1}_1\gamma_2]$. There exist $\dim_k M^g_a=1$ for all $a\in T(M^g)_0$ and $M^g(\gamma_i)=\lambda_i$. 
\end{example}
 
We define a push-down functor     $\pi_\lambda:$ rep$_k(T(n), \Omega)\rightarrow$ rep$_k(K(n)); M \mapsto (\pi_\lambda(M)_j, \pi_\lambda(M)(\gamma_i))$ factoring through $T(n)/G$, that is, $\pi_\lambda(M^g)=\pi_\lambda(M)$, $\forall g\in G$, where

\begin{equation*} 
\pi_\lambda(M)_j:=\bigoplus _{\pi(y)=j}M_y,
\end{equation*}

\begin{equation*} 
\pi_\lambda (M)(\gamma_i):=\bigoplus _{\pi(\beta)=\gamma_i} M(\beta):\pi_\lambda (M)_1\rightarrow \pi_\lambda (M)_2, j\in \{1, 2\},1\leq i\leq n.
\end{equation*}
The  functor $\pi_{\lambda}$ is  exact \cite[3.2]{Bongartz}. Since the category $\rep_k(K(n))$ is equivalent to the category  $\modd \cK_n$ of finite-dimensional modules for path algebra $\cK_n$ of generalized Kronecker quiver $K(n)$,   we sometimes say that a representation  $M\in\rep_k(K(n))$  is a module of $\mathcal{K}_n$. We call modules in $\rep_k(T(n), \Omega)$  the \textit{graded Kronecker modules} (or \textit{graded modules}, or simply \textit{modules}).  We often use  $\modd(T(n),\Omega)$ to denote  those modules.  We have the following
 
\begin{thm}\cite[Theorem 7.1]{Daniel}\label{push-down}
The following statements hold.
\begin{enumerate}

\item $\pi_\lambda$ sends indecomposable representations in $\rep_k(T(n), \Omega)$ to indecomposable representations  in $\rep_k(K(n))$.

\item If $M\in\rep_k(T(n), \Omega)$ is indecomposable, then $\pi_\lambda(M)\cong \pi_\lambda(N)$ if and only if $M^g\cong N$ for some $g\in G$.

\item $\pi_\lambda$ sends $AR$-sequences to $AR$-sequences and $\pi_\lambda$ commutes with the Auslander-Reiten translations.

\item If $M\in\rep_k(T(n), \Omega)$ is indecomposable in a component $\cD$ with $\pi_\lambda(M)$ in a component $\cC$, then $\pi_\lambda$ induces a covering $\cD\rightarrow\cC$ of translation quivers.

\end{enumerate}
\end{thm}

\begin{cor}\cite[Corollary 7.2]{Daniel}\label{component}
Let $\cD$ be a regular component of $\modd(T(n),\Omega)$. Then the covering $\cD\rightarrow\pi_\lambda(\cD)$ is an isomorphism of translation quivers. In particular, $\cD$ is of type $\mathbb{Z}A_{\infty}$.

\end{cor}

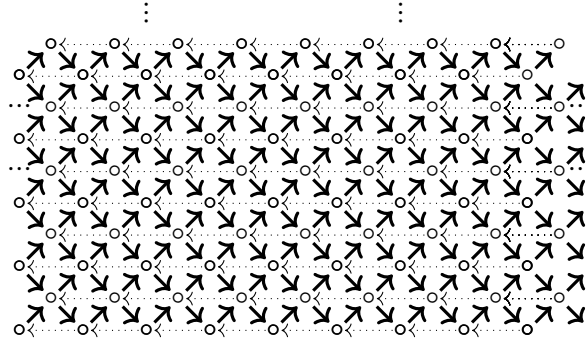
\begin{figure}[!h]

\begin{center}

\begin{tikzpicture}[very thick,scale=0.7]

                    [every node/.style={fill, circle, inner sep = 1pt}]

%%%%%%%%%%%%%%%%%%%%%%%%%%%%%% Parameter %%%%%%%%%%%%%%%%%%%%%%%%%%%%%%%%%%%

\def \n {8} % #Knoten Reihe  - 1

\def \m {4} % #Knoten Spalte - 1

\def \translation {1} % 1 Für Translation

\def \ab {0.15} % Abstand Pfeil und Knoten

\def \Pab {0.6} % Halber Abstand Horizontal

\def \lcone {1} % 1 für linken Kegel

\def \ldist {3} % Anzahl der quasi-einfachen die eingeschlossen werden - 1

\def \lcolor {red} % white für keine Farbe

\def \rcone {1} % 1 für rechten Kegel

\def \rdist {3} % Anzahl der quasi-einfachen die eingeschlossen werden - 1

\def \rcolor {red} %  white für keine Farbe

\def \llcone {0} % 1 für einen zweiten linken Kegel rechts von lcone

\def \lldist {4} % Anzahl der quasi-einfachen die eingeschlossen werden - 1

\def \rrcone {0} %1 für einen zweiten rechten Kegel links von rcone

\def \rrdist {4} % Anzahl der quasi-einfachen die eingeschlossen werden - 1

%%%%%%%%%%%%%%%%%%%%%%%%%%%% Quellcode %%%%%%%%%%%%%%%%%%%%%%%%%%%%%%%%%%%%%%

\foreach \a in {0,...,\n}{

\foreach \b in {0,...,\m}{

    \ifthenelse{\a = \n \and \b < \m}{

   %\node[color=black] ({\a,\b}) at ({\a*2*\Pab+\Pab},{\b*2*\Pab+\Pab}) {$\circ$};

   \node[color=black] ({\a,\b,5})at ({\a*2*\Pab},{\b*2*\Pab}) {$\circ$};

     }

     {

      \ifthenelse{\b = \m \and \a < \n}{

      \node[color=black] ({\a,\b}) at ({\a*2*\Pab+\Pab},{\b*2*\Pab+\Pab}) {$\circ$};

      \node[color=black] ({\a,\b,5})at ({\a*2*\Pab},{\b*2*\Pab}) {$\circ$};

      }

      {

      \ifthenelse{\b = \m \and \a = \n}

     {\node[color=black] ({\a,\b,5})at ({\a*2*\Pab},{\b*2*\Pab}) {$\circ$};}

    {\node[color=black] ({\a,\b}) at ({\a*2*\Pab+\Pab},{\b*2*\Pab+\Pab}) {$\circ$};

    \node[color=black] ({\a,\b,5})at ({\a*2*\Pab},{\b*2*\Pab}) {$\circ$};

      }

      }

      }

    }

    }

\foreach \s in {0,...,\n}{

\foreach \t in {0,...,\m}{  

 \ifthenelse{\t = \m \and \s < \n}{

    \draw[->] (\s*2*\Pab+\ab,\t*2*\Pab+\ab) to (\s*2*\Pab+\Pab-\ab,\t*2*\Pab+\Pab-\ab); 

    \draw[->] (\s*2*\Pab+\Pab+\ab,\t*2*\Pab+\Pab-\ab) to (\s*2*\Pab+2*\Pab-\ab,\t*2*\Pab+\ab); 

 }{

   \ifthenelse{\s = \n \and \t < \m}{

    %\draw[->] (\s*2*\Pab+\Pab+\ab,\t*2*\Pab+\Pab+\ab) to (\s*2*\Pab+2*\Pab-\ab,\t*2*\Pab+2*\Pab-\ab);

   %\draw[->] (\s*2*\Pab+\ab,\t*2*\Pab+2*\Pab-\ab) to (\s*2*\Pab+\Pab-\ab,\t*2*\Pab+\Pab+\ab);  

   %\draw[->] (\s*2*\Pab+\ab,\t*2*\Pab+\ab) to (\s*2*\Pab+\Pab-\ab,\t*2*\Pab+\Pab-\ab); 

   }

  {

  \ifthenelse{\s = \n \and \t = \m}{

    }{

   \draw[->] (\s*2*\Pab+\ab,\t*2*\Pab+\ab) to (\s*2*\Pab+\Pab-\ab,\t*2*\Pab+\Pab-\ab); 

   \draw[->] (\s*2*\Pab+\Pab+\ab,\t*2*\Pab+\Pab+\ab) to (\s*2*\Pab+2*\Pab-\ab,\t*2*\Pab+2*\Pab-\ab);

   \draw[->] (\s*2*\Pab+\ab,\t*2*\Pab+2*\Pab-\ab) to (\s*2*\Pab+\Pab-\ab,\t*2*\Pab+\Pab+\ab); 

   \draw[->] (\s*2*\Pab+\Pab+\ab,\t*2*\Pab+\Pab-\ab) to (\s*2*\Pab+2*\Pab-\ab,\t*2*\Pab+\ab);    

   }

   }

   }

    }

    }

\draw[->] (\n*2*\Pab+\ab,\m*\Pab+2*\Pab+\Pab+\Pab+\ab) to (\n*2*\Pab+\Pab-\ab,\m*\Pab+2*\Pab+\Pab+\Pab+\Pab-\ab);

\ifthenelse{\isodd{\m}}

%% IF

 { 

  \node[color=black] (Dots1) at (0,\m*\Pab+2*\Pab+\Pab) {$\cdots$};

  \node[color=black] (Dots2) at (1+\n*2*\Pab,\m*\Pab+2*\Pab+\Pab) {$\cdots$};

   \ifthenelse{\isodd{\n}}{

  \node[color=black] (Dots3) at (-0.2*\n*2*\Pab,2*\m*\Pab+3*\Pab) {$\vdots$};}

  {\node[color=black] (Dots4) at (0.75*\n*2*\Pab,2*\m*\Pab+2*\Pab) {$\vdots$};} 

  }

%% Else

  {

  \node[color=black] (Dots1) at (0,\m*\Pab+\Pab) {$\cdots$};

  \node[color=black] (Dots2) at (1+\n*2*\Pab,\m*\Pab+\Pab) {$\cdots$};

  \ifthenelse{\isodd{\n}}{

  \node[color=black] (Dots3) at (0*\n*2*\Pab,2*\m*\Pab+3*\Pab) {$\vdots$};}

  {\node[color=black] (Dots4) at (0.25*\n*2*\Pab,2*\m*\Pab+2*\Pab) {$\vdots$};}

  }

\ifthenelse{\translation = 1}{

   \foreach \s in {0,...,\n}{

   \foreach \t in {0,...,\m}{ 

   \ifthenelse{\s = 0}{}{

      \ifthenelse{\s = \n}{\draw[->,dotted,thin] (\s*2*\Pab-\ab,\t*2*\Pab) to (\s*2*\Pab-2*\Pab+\ab,\t*2*\Pab); }{

   \draw[->,dotted,thin] (\s*2*\Pab-\ab,\t*2*\Pab) to (\s*2*\Pab-2*\Pab+\ab,\t*2*\Pab); 

   \draw[->,dotted,thin] (\s*2*\Pab-\ab+\Pab,\t*2*\Pab+\Pab) to (\s*2*\Pab-2*\Pab+\Pab+\ab,\t*2*\Pab+\Pab); 

\draw[->,dotted,thin] (\n*2*\Pab-\ab+\Pab,\t*2*\Pab+\Pab) to (\n*2*\Pab-2*\Pab+\Pab+\ab,\t*2*\Pab+\Pab); 

   }

   }}

}}

{}  %ELSE

                                                                                                                                                                                             \end{tikzpicture}

\end{center}

\caption{Regular component  $\mathbb{Z}A_{\infty}.$}

\label{Fig:RCGM}

\end{figure}

\subsection{Graded Kronecker modules}

Let $M\in$ mod$(T(n) ,\Omega)$ be  indecomposable.  Then $T(M)=$ supp$(M)$ is a finite tree, we can  define

\begin{center}
$d(M)=d(T(M))$,  $c(M)=c(T(M))$ and $r(M)=r(T(M))$. 
\end{center}
We call them   \textit{diameter}, \textit{center},  and  \textit{radius}  of $M$, respectively  \cite[2.4]{Claus}.  We say that 
\begin{enumerate}
\item  $M$ is said to be a \textit{sink}  module,        provided diameter walks of $T(M)$ start and end in sinks;

\item $M$ is said to be a  \textit{source} module,  provided diameter walks of $T(M)$ start and end in sources;

\item $M$ is said to be a \textit{flow} module,  provided length of diameter walk of $T(M)$ is odd.
 
\end{enumerate}

\subsection{Index and Balls}\label{Ball}

 For any $ x, c, c_1, c_2\in T(n)_0$ and $r\in \mathbb{N}$ \cite[2.3]{Claus}, let
\begin{center}
$B_r(c):=\{x\in T(n)_0 \mid d(x,c) \leq r\}$, 

$B_r(c_1,c_2):=\{x \in T(n)_0\mid d(x, c_i)\leq r\}$, $i\in \{1,2\}$.
\end{center}
We call such sets $B_r(c)$ and $B_r(c_1, c_2)$ the \textit{balls} with radius $r$ and with center $c$ or $\{c_1, c_2\}$, respectively. We sometimes take these balls $B_r(c)$ and $B_r(c_1, c_2)$ as subquivers of $T(n)$.

 \section{ graded modules of constant rank}\label{Matrice}

Let $M\in \modd(T(n),\Omega)$ be an  indecomposable module, and let  $N=\pi_\lambda(M)$, $n\geq 3$. We now focus on the modules $M$ and $N$.   Let  $c$ (or $\{c_1,c_2\}$) and $r$ be the center and radius of $T(M)$, respectively.     We have $\supp(M)\subseteq B_r(c)$ ( or $ \supp(M)\subseteq B_{r+1}(c_1,c_2))$. Hence we can always find a minimal ball $B(M)$ such that $\supp(M)\subseteq B(M)$. Clearly, the  sets $T(M)$ and $B(M)$ have the same radius and same center.    Note that

\begin{equation*} 
N_1=\pi_\lambda(M)_1=\bigoplus _{\pi(y)=1}M_{y},  N_2=\pi_\lambda(M)_2=\bigoplus _{\pi(y)=2}M_{y},  
\end{equation*}

\begin{equation*} 
N(\gamma_i)=\bigoplus _{\pi(\gamma^j_i)=\gamma_i} M(\gamma^j_i):\pi_\lambda (M)_1\rightarrow \pi_\lambda (M)_2, 
\end{equation*}
where $y\in T(n)_0$,  $j\in \mathbb{N}$,   and $ i\in\{1,2,\cdots,n\}$. For convenience, we let $M(\gamma^j_i)=\bg^j_i$ and $N(\gamma_i)=\bg_i$ for all $1\leq i\leq n$.  By abusing the notation, we sometimes use $\bg^j_i$ to denote its  corresponding matrix (linear map, or arrow $\gamma^j_i$) in $T(M)$, and we let $\bg^j_i.v=\bg^j_i v$, where the latter $\bg^j_i$ is a matrix and $v\in M_{s(\bg^j_i)}$.

 Recall that  $\mathcal{N}_M(a)=\mathcal{N}(a) \cap T(M)_0$ for any $a\in T(M)_0$.    Suppose that $\supp(M)=B(M)$. If $p=(a_0,a_1,\cdots,a_{d(M)})$ is a diameter walk of $T(M)$, then we  have $\mid\cN_M(a_{i})\mid=n$ when $d(M)\geq 2, 1\leq i \leq d(M)-1$.
 We define
\begin{center}
 $ \Gamma_M:=\sum^n_{i=1}(\bigoplus _{\pi(\gamma^j_i)=\gamma_i} M(\gamma^j_i))=\sum^n_{i=1}(\bigoplus_{\pi(\gamma^j_i)=\gamma_i} \bg^j_i)=\sum^n_{i=1}N(\gamma_i)=\sum^n_{i=1}\bar{\gamma}_i$.   
\end{center} 

  We now fix two bases of the vector spaces $\{M_y\mid \pi(y)=1\}$ and $\{M_y\mid \pi(y)=2\}$, and we want to discuss the matrix $\Gamma_M$ when $T(M)=B(M)$.  We use induction on the diameter $d(M)$ of $T(M)$, and we always assume that the walk $p=(a_0,a_1,\cdots, a_{d(M)})$ is a diameter walk of $T(M)$. Assume that $d(M)>0$.  If $d(M)=1$, then  $M$ is a flow module and there only exists one non-trivial arrow $\gamma^j_i:a_0\rightarrow a_1\in T(M)_1$ with $c(M)=\{a_0,a_1\}$ and $r(M)=0$.  Note that the module $M$ has nothing to do with the indices $i,j$ in $\gamma^j_i$. Then we can  start from  smaller number, and let $i=j=1$ (we do the same thing in the following).  Hence we have 
\begin{center}
$\Gamma_M=\bg_1=\bg^1_1$, where $N_1=M_{s(\bg^1_1)}$.
\end{center}
Suppose that $d(M)=2$. Then $M$ is a sink module or a source module, and we get $\mid \cN_M(a_1)\mid =n$ by $T(M)=B_1(a_1)$, i.e.,   there exist non-trivial maps $\bg^1_1, \bg^1_2,\cdots, \bg^1_n$ in $M$. If  $M$ is a sink module,  then we have $s(\bg^1_i)=s(\bg^1_j)$ for $1\leq i\neq j\leq n$ and 
\begin{center}
 $N_1=M_{s(\bg^1_i)}$, $N_2=\bigoplus^n_{i=1} M_{t(\bg^1_i)}$.
\end{center}
Hence we have 

\begin{center}
$\bg_1= \begin{bmatrix}
\bg^1_1 \\ 
0\\
\vdots \\
0
\end{bmatrix}$, \quad $\bg_2= \begin{bmatrix}
0 \\ 
\bg^1_2\\
0\\
\vdots \\
0
\end{bmatrix} $, \quad $\cdots$ ,  \quad $\bg_n= \begin{bmatrix}
 
0\\
\vdots \\
0\\
\bg^1_n
\end{bmatrix} $. 
\end{center}
Then $\Gamma_M$ would be a block matrix as the following

\begin{center}
\begin{equation*}
\Gamma_M=\sum^{n}_{i=1}\bg_i=\begin{bmatrix}
\bar{\gamma}^1_{1}\\
\bar{\gamma}^1_{2}\\
\vdots \\
\bar{\gamma}^1_{n}
\end{bmatrix}, \text{ where }t(\bg^1_i)=a_1, 1\leq i\leq n.
\end{equation*}
\end{center} 
When $M$ is a source module, similarly, we have
\begin{center}
$\Gamma_M= \sum^{n}_{i=1}\bar{\gamma_i}=\begin{bmatrix}
\bar{\gamma}^1_{1} &  \bar{\gamma}^1_{2}& \cdots & \bar{\gamma}^1_{n}
\end{bmatrix}$, where $s(\bg^1_i)=a_1, 1\leq i\leq n$.
\end{center}
Note that we only need to discuss the matrix $\Gamma_M$ when $M$ is a sink  module,  and we  can  take it as a transpose of a matrix   when $M$ is a source module. From now on,  we always assume that $M$ is a sink module when its diameter $d(M)$ is even.  If $d(M)=3$, then $M$ is a flow module, and we have $\mid \cN_M(a_1)\mid=\mid \cN_M(a_2)\mid =n$. That is to say, there exist non-trivial maps $\bg^1_1,\cdots, \bg^1_n$ and $\bg^2_1,\cdots,\bg^2_{n-1}$ with $s(\bg^1_i)=a_1$ and $t(\bg^2_j)=a_2$, $1\leq i\leq n$, $1\leq j\leq n-1$. Then 

\begin{center}

$\Gamma_M=\sum^n_{i=1}\bg_i: M_{s(\bg^1_i)}\bigoplus (\bigoplus^{n-1}_{j=1}M_{s(\bg^2_j)})\rightarrow (\bigoplus^n_{i=1}M_{t(\bg^1_i)})\bigoplus M_{t(\bg^2_2)}$.
\end{center}
Suppose that  $t(\bg^1_n)=a_2$, i.e.,  $\bg^1_n= M(\gamma^1_n) : M_{a_1}\rightarrow M_{a_2}$. Then we have

\begin{center}
\begin{equation*}
\Gamma_M=\begin{bmatrix}
\bar{\gamma}^1_{1} & 0 & \cdots & 0\\
\vdots &  \vdots & \ddots & \vdots\\
\bar{\gamma}^1_{n-1} & 0 & \cdots & 0\\
\bar{\gamma}^1_{n} & \bar{\gamma}^2_{1} & \cdots & \bar{\gamma}^2_{n-1}
\end{bmatrix}.
\end{equation*}
\end{center}
Suppose that $d(M)=4$. Without loss of generality, we assume that $\mid \cN_M(t(\bg^2_i))\mid=n$ and $c(M)=t(\bg^2_i)=a_2, 1\leq i\leq n-1$. Then we have

\begin{center}
$\Gamma_M= M_{s(\bg^1_i)}\bigoplus (\bigoplus^{n-1}_{j=1}M_{s(\bg^2_j)})\rightarrow (\bigoplus^n_{i=1}M_{t(\bg^1_i)})\bigoplus M_{t(\bg^2_2)}\bigoplus (\bigoplus^{n-1}_{i=1} M_{t(\bg^3_i)}) \bigoplus \cdots \bigoplus (\bigoplus^{n-1}_i M(t(\bg^{n+1}_i))$,
\end{center}
where $s(\bg^{i+2}_j)=s(\bg^2_i)$, $1\leq i\leq n-1$, $1\leq j\neq i\leq n$. That is,

\begin{center}
$\Gamma_M=\begin{bmatrix}
\bar{\gamma}^1_{1} & 0 & 0 &  \cdots & 0\\
\vdots &  \vdots & \vdots & \ddots & \vdots\\
\bar{\gamma}^1_{n-1} & 0 & 0&  \cdots & 0\\
\bar{\gamma}^1_{n} & \bar{\gamma}^2_{1} & \bg^2_2&  \cdots & \bar{\gamma}^2_{n-1}\\
0 & \bg^3_2 & 0&  \cdots & 0\\
0 & \bg^3_3 & 0&  \cdots & 0 \\
\vdots & \vdots &  \vdots& \ddots & \vdots \\
0 & \bg^3_n & 0 & \cdots & 0 \\
0 & 0& \bg^4_1 & \cdots & 0\\
\vdots & \vdots& \vdots & \ddots & 0\\
0 & 0 & \bg^4_n & \cdots & 0\\
0& 0 &  0 & \cdots & \bg^{n+1}_1\\
\vdots&  \vdots & \vdots & \ddots& \vdots\\
0& 0&  0& \cdots& \bg^{n+1}_{n}

\end{bmatrix}$
\end{center}
Clearly, we can continue adding the vector spaces in $N_1$ or $N_2$ and  finding the matrices $\bg^j_i$ in the block matrix $\Gamma_M$ using the induction on the diameter $d(M)$. Finally, we can see that every row and every column have $n$ matrices in $\Gamma_M$ except the leaves in $T(M)$ when $T(M)=B(M)$.

Let   $\cA_i=\{\bg^j_i\mid \gamma^j_i\in T(M)_1 \}$, i.e., the set $\cA_i$ is the set of direct summands of the map $\bg_i$, $1\leq i\leq n$. Then 

\begin{proposition}\label{position}
Let $M\in \modd(T(n),\Omega)$ be an indecomposable module, and let $(a_0,a_1,\cdots,a_{d(M)})$ be a diameter walk of $T(M)$ with $d(M)\geq 2$. Suppose that $T(M)=B(M)$. Then the matrix $\Gamma_M$ would look like the following
\begin{center}

$\Gamma_M=\sum^n_{i=1}\bg_i = \begin{bmatrix}
\bg^1_1 & 0\\
\bg^1_2 & 0\\
\vdots & 0\\
\bg^1_{n-1} & 0 \\
\bg^1_n & \Gamma'_M \\

\end{bmatrix}$ or  $\Gamma_M=\sum^n_{i=1}\bg_i = \begin{bmatrix}
\bg^1_1 & \bg^1_2 & \cdots & \bg^1_{n-1} & \bg^1_n \\
0 & 0 & \cdots & 0  & \Gamma'_M 
\end{bmatrix}$,
\end{center}
where  $\Gamma'_M=\sum^{n}_i \bg'_i$, $\bg'_i=\bigoplus_{\bg^j_i\in \cA'_i}\bg^j_i$, $\cA'_i=\cA_i\setminus \{\bg^1_i\}$,   $\bg^1_i$ is the linear map between  $M_{a_1}$ and its neighbours $M_b$, $b\in \cN_M(a_1)$, $1\leq i\leq n$.

\end{proposition}

We recall an old result.

\begin{Lemma}\cite [Theorem 19, $(8.1), (8.2), (8.7)$]  {George}\label{matrix}  Ranks of Partitioned Matrices. For matrices over an arbitrary
field:

\begin{center}
 $\rk \begin{bmatrix}
A & B
\end{bmatrix}  \geq  \rk A$,  $\rk \begin{bmatrix}
A \\
B
\end{bmatrix}\geq \rk A$,

 $\rk \begin{bmatrix}
X & 0 \\
S & Y
\end{bmatrix}\geq \rk X+ \rk Y$.

\end{center}

\end{Lemma}

\quad

 For any $\alpha \in k^n \setminus \{0\}$ and any $M\in$  mod $\mathcal{K}_n$, we define 
 \begin{center}
  $x_{\alpha}:= \alpha_1\gamma_1+\cdots+ \alpha_n\gamma_n$.
 \end{center}
We denote by
\begin{center}
 $x^{\alpha}_{M}: M \rightarrow M$
\end{center}
 the linear operator associated to $x_{\alpha}$: the left multiplication with $x_{\alpha}$.

\begin{Definition}
Let $M\in\modd \mathcal{K}_n$ be an indecomposable module.    Then $M$ is said to have the \textit{equal kernels property},  provided  $\ker x^{\alpha}_{M}$ is independent of $\alpha$ for all $\alpha \in k^n \setminus \{0\}$. Further, $M$ is said to have the \textit{equal images property}, provided $\im x^{\alpha}_{M}$ is independent of $\alpha$.
\end{Definition}

Let $M\in$ mod $\mathcal{K}_n$ be indecomposable and non-simple. General theory tells us that 
\begin{center}
$\bigcap^n_{i=1} \ker(x^{e_i}_{M}) = M_2$ and $\sum^n_{i=1} \im(x^{e_i}_{M}) = M_2$,
\end{center}
where $e_i$ is the $i$-th canonical coordinate, i.e., the $i$-th position is $1$ and others are zero. In particular,  module $M$ has the equal kernels property if and only if $M_2 = \ker(x^{\alpha}_{M})$ for all $\alpha \in k^n \setminus \{0\}$, and has the equal images property if and only if $M_2 = \im(x^{\alpha}_{M})$ for all non-zero $\alpha$ \cite[Lemma 2.20]{Daniel}. Hence   we  can define two sets

\[ \EKP := \{M \in \modd \mathcal{K}_n \text{ does not have direct summands of type } S(1) \mid  \forall \alpha \in k^n \setminus (0): \ker(x^{\alpha}_{M}) =  M_2 \},\]

\[ \EIP := \{M \in  \modd \mathcal{K}_n \text{ does not have direct summands of type }S(2) \mid  \forall \alpha \in k^n \setminus (0): \im(x^{\alpha}_{M}) =  M_2 \}.\]

Then we have the definition.

\begin{Definition}
Let $N\in \text{mod }\mathcal{K}_n$. We say that the module   $N$ is of \textit{constant rank} if the rank of $x^{\alpha}_N$ is independent of $\alpha\in k^n\setminus \{0\}$. Similarly,  an indecomposable module $M\in \modd(T(n), \Omega)$ is said to be of constant rank if $\pi_\lambda(M)$ is of constant rank.                                                                                                                                                                         
\end{Definition}

Let $M\in \modd(T(n),\Omega)$ be an indecomposable module, and let  $\Gamma^{\alpha}_M=x^{\alpha}_{\pi_\lambda(M)}=\sum^{n}_{i=1}\alpha_i \bg_i$ for any $\alpha\in k^n\setminus\{0\}$. According to Proposition \ref{position}, we have

\begin{center}

$\Gamma^{\alpha}_M=\sum^n_{i=1}\alpha_i\bg_i = \begin{bmatrix}
\alpha_1\bg^1_1 & 0\\
\alpha_2\bg^1_2 & 0\\
\vdots & 0\\
\alpha_{n-1}\bg^1_{n-1} & 0 \\
\alpha_n\bg^1_n & \Gamma^{\alpha'}_M \\

\end{bmatrix}$ or  $\Gamma^{\alpha}_M= \begin{bmatrix}
\alpha_1\bg^1_1 & \alpha_2\bg^1_2 & \cdots & \alpha_{n-1}\bg^1_{n-1} & \alpha_n\bg^1_n \\
0 & 0 & \cdots & 0  & \Gamma^{\alpha'}_M 
\end{bmatrix}$,
\end{center}
where  $\Gamma^{\alpha'}_M=\sum^{n}_i \alpha_i \bg'_i$, $\bg'_i=\bigoplus_{\bg^j_i\in \cA'_i}\bg^j_i$, $\cA'_i=\cA_i\setminus \{\bg^1_i\}$,   $\bg^1_i$ is the linear map between  $M_{a_1}$ and its neighbours $M_b$, $b\in \cN_M(a_1)$, $1\leq i\leq n$. If we let $\alpha_i\bg^j_i=\beta^j_i$ and $\alpha_i\bg_i=\beta_i$, then we can see that   $\rk \Gamma^{\alpha}_M=\rk(\sum^n_{i=1}\alpha_i\bg_i)=\rk(\sum^n_{i=1}\beta_i) =\rk \Gamma_{M'}=\rk \Gamma_M$, where 
$\Gamma_{M'}=\begin{bmatrix}
\beta^1_1 & 0\\
\beta^1_2 & 0\\
\vdots & 0\\
\beta^1_{n-1} & 0 \\
\beta^1_n & \Gamma_{M''} \\

\end{bmatrix}$ and $\Gamma_{M''}= \Gamma^{\alpha'}_M$. 
  Hence it is enough to focus on the matrix $\Gamma_M$ when each  $\alpha_i$ is not zero for all $1\leq i\leq n$.  We have the subcategory
\begin{center}
$ \CR:= \{M\in \modd \mathcal{K}_n\mid \text{there is } r_M\in \mathbb{N}_0 \text{ such that rk } x^{\alpha}_M=r_M \text{ for all } \alpha\in k^n\setminus \{0\} \}$.
\end{center}
 There exist  $\EKP \cap \EIP=(0)$ and $\EIP \cup \EKP\subseteq \CR$ \cite[2.4]{Julia}.  On the other hand, we have

\begin{center}
$\Inj := \{M\in\modd(T(n), \Omega)\mid \forall \delta \in (T(n))_1: M(\delta)$ is injective $\}$,

$\Sur:=\{M\in \modd(T(n), \Omega)\mid \forall \delta \in (T(n))_1: M(\delta)$ is surjective $\}.$ 

$\GR:=\{M\in \modd(T(n),\Omega)\mid \pi_\lambda(M)\text{ is of constant rank}\}$.

\end{center}
Let $M\in$ mod$(T(n), \Omega)$ be regular indecomposable. According to \cite[Theorem 8.1]{Daniel}, we know that $M\in\Inj$ if and only if $M\in\EKP $, and we have the following conclusions.

\begin{Theorem} \cite[Theorem 3.7]{Carlson}\label{DS}
 Let $M\in \modd(T(n),\Omega)$ be of constant rank. Then any direct summand of $M$ is also of  constant rank. 

\end{Theorem}

Let $M\in \modd(T(n),\Omega)$ be an indecomposable module, and  let $N=\pi_\lambda(M)$. From now on, we always assume that  $M$  is indecomposable and  of constant rank if we do not mention it specifically in this section. Suppose that $N$ is not zero or simple. Then we have $d(M)\geq 1 $.  Let 
\begin{center}
$\cS_i=\{s(\gamma^j_i)\mid \pi(\gamma^j_i)=\gamma_i \}$, $\cT_i=\{t(\gamma^j_i)\mid \pi(\gamma^j_i)=\gamma_i\}$, $1\leq i\leq n$.

\end{center}

Then we have

\begin{Lemma}\label{ddpoint}
Suppose that $M\in \modd(T(n),\Omega)$ is an  indecomposable and of constant rank module. Then  for the sets $\cS_i, \cT_i$ and every arrow $\gamma^q_j\in T(M)_1$ there exist
\begin{center}
$ s(\gamma^{q}_j)\in \cS_i$ or  $t(\gamma^{q}_j)\in \cT_i, i\neq j\in\{1,2,\cdots,n\}$.
\end{center}

\end{Lemma}
\begin{proof}
Let  $N=\pi_\lambda{(M)}$.   Suppose that there exists a non-trivial map  $M(\gamma^u_j): M_{s(\gamma^u_j)} \rightarrow M_{t(\gamma^u_j)}$  in $\cA_j$ for some  $i$ ($\neq j$) such that $s(\gamma^u_j)\nin \cS_i$ and $t(\gamma^u_j)\nin \cT_i$, that is, there is no   direct summand  $\bg^u_j$ of $\bg_j$ such that it is located in the same row or column with some direct summand $\bg^v_i$ of $\bg_i$ in the block matrix $\bg_i+\bg_j$.  According to Proposition \ref{position}, we have
\begin{equation}\label{rk}
\rk(\bg_i+\bg_j)=\rk(\bg_i+(\bigoplus_{\pi(\gamma^p_j)=\gamma_j} M(\gamma^q_j))\geq \rk \begin{bmatrix}
\bg_i & 0\\
0& \bg^u_j
\end{bmatrix}>\rk \bg_i.
\end{equation}
Clearly, the inequality $(\ref{rk})$ cannot happen because module $N$ is of constant rank and  $\rk(\bg_i+\bg_j) = \rk \bg_i=c$ for some $c\in \mathbb{N}$.
\end{proof}

In fact, the maps in $N=\pi_\lambda(M)$ satisfy the following conditions when $N$ is of constant rank.

\begin{Lemma}\cite[Lemma 3.1]{West}\label{KI} Let $N\in \modd\cK_n$ be a constant rank module. Then the maps in module $N$ satisfy 

\begin{center}
$N(\gamma_i)(\ker N(\gamma_j))\subseteq \im N(\gamma_j)$,
\end{center}
for any $i,j\in \{1,\cdots,n\}$.
\end{Lemma}

When $N=\pi_\lambda(M)$ and $M$ is indecomposable and of constant rank, Lemma \ref{KI} is interpreted as  the following 

\begin{Lemma}\label{tt}
Let $\bg^p_i\in \cA_i$, and let $\bg^{p}_j,\bg^{p'}_j\in \cA_j$ $i\neq j$. Suppose that there exist   $s(\bg^p_i)=s(\bg^p_j)$ and $t(\bg^p_i)=t(\bg^{p'}_j)$. Then we have 
\begin{center}
$\bg^p_i(\ker \bg^{p}_j)\subseteq \im \bg^{p'}_j$.
\end{center}

\end{Lemma}

\begin{proof}
Suppose that $\rk x^N_\alpha =m$.  Recall that
\begin{center}
$\bg_i=N(\gamma_i)=\bigoplus _{\pi(\gamma^x_i)=\gamma_i} M(\gamma^x_i)=\bigoplus _{\pi(\gamma^x_i)=\gamma_i}\bg^x_i$, and $\bg_j=N(\gamma_j)=\bigoplus _{\pi(\gamma^y_j)=\gamma_j} M(\gamma^y_j)=\bigoplus _{\pi(\gamma^y_i)=\bg_j}\bg^y_j $. 
\end{center}
Then $\ker \bg_j=\ker (\bigoplus _{\pi(\gamma^y_j)=\gamma_j}M(\gamma^y_j))=\bigoplus _{\pi(\gamma^y_j)=\gamma_j}\ker M(\gamma^y_j)$. Hence we have 
\begin{equation}
\begin{split}
\bg_i(\ker \bg_j)& =( \bigoplus _{\gamma^x_i\in \cA_i} M(\gamma^x_i))(\bigoplus _{\gamma^y_j\in \cA_j}\ker M(\gamma^y_j)) \\& = \bigoplus_{s(\bg^p_i)=s(\bg^{p}_j), \bg^p_i\in \cA_i,\bg^p_j\in\cA_j}\bg^p_i(\ker (\bg^{p}_j)) \\ &  \subseteq \bigoplus _{t(\gamma^p_i)=t(\gamma^p_j), \bg^p_i\in \cA_i,\bg^p_j\in\cA_j} \im M(\gamma^p_j) \\ & \subseteq\im (\bigoplus _{\pi(\gamma^p_j)=\gamma_j, \bg^p_j\in \cA_j} M(\gamma^p_j)) \\& = \im \bg_j.
\end{split}
\end{equation}
Otherwise, we assume that there exist some $\bg^p_i \in \cA_i$, and $\bg^{p}_j,\bg^{p'}_j\in \cA_j$ such that 
\begin{center}
$\bg^p_i(\ker \bg^{p}_j)\nsubseteq \im \bg^{p'}_j$.
\end{center}
Note that we only have one direct summand $\bg^{p'}_j$ of $\bg_j$ satisfying $t(\bg^p_i)=t(\bg^{p'}_j)$,  that is to say, $\bg_i(\ker \bg_j)\nsubseteq\im \bg_j$. We can see that module $N$ is not of constant   rank by Lemma \ref{KI}, this is a contradiction. Finally, we have 
\begin{center}
 $N(\gamma_i)(\ker N(\gamma_j))\subseteq \im N(\gamma_j) \Leftrightarrow N(\gamma^p_i)(\ker \gamma^{p}_j)\subseteq \im N(\gamma^{p'}_j)$, 
\end{center}
for any direct summands $\bg^p_i\in \cA_i, \bg^p_j,\bg^{p'}_j\in \cA_j$ with $s(\bg^p_i)=s(\bg^{p}_j)$ and $t(\bg^p_i)=t(\bg^{p'}_j)$.

\end{proof}

We also have an easy result for the diameter walk of $T(M)$.

\begin{Lemma}
Let $M\in \modd(T(n),\Omega)$.     Suppose that $M$ is indecomposable and of constant rank. Then $d(M)=0$  or $d(M)\geq 2$.

\end{Lemma}
\begin{proof}
When $N=\pi_\lambda(M)$ is a zero or simple module,  we clearly have $d(M)=0$. We now suppose that $N$ is not a zero or simple module. Then we get $d(M)\geq 1$. If $d(M)=1$, then there only exists  one  diameter walk $(a_0, a_1)$ and one non-trivial map $M(\gamma^1_1): M_{a_0}\rightarrow M_{a_1}$ in $T(M)$. Since $n\geq 3$  for the tree $T(n)$, this cannot happen. Hence we have $d(M)\geq 2$.

\end{proof}

From now on, we always assume that $d(M)\geq 2$ when we talk about an indecomposable module $M\in \GR$.

\begin{Lemma}\label{point}
 Let $(a_0,a_1,\cdots,a_m)$ be  a  walk of $T(M), m\geq 2$. Suppose that $a_0$ is a leaf. Then we have $\mid \mathcal{N}_M(a_1)\mid=n$, $n\geq 3$.
\end{Lemma}
\begin{proof}
Since    $M$ is of constant rank, we have   $\rk \bg_1=\rk \bg_2=\cdots=\rk\bg_n$ in $N$.   Suppose that there exists a walk $(a_0,a_1,\cdots,a_m)$ with $a_0$ being a leaf in $T(M)$ such that  $\mid \mathcal{N}_M(a_1) \mid <n$. We first assume that  $a_0$ is a sink. It indicates that there exists at least one  $i$ such that $a_1\nin \cS_i$,  that is, there is no a   non-trivial direct summand  $\bg^p_i: M_{a_1}\rightarrow M_{a'}$ in $T(M)$ for all $\bg^p_i\in \cA_i$, where $ a'\in \mathcal{N}_M(a_1)$. Since $a_0$ is a leaf, it also says that there does not exist a direct summand $\bg^{q}_i\in \cA_i$  such that $t(\bg^{q}_i)=a_{0}$, that is, $a_0\nin \cT_i$. On the other hand,  there  exists at least  one map $\bg^{w}_{j}: M_{a_1}\rightarrow M_{a_0} $ that is non-trivial for some $\bg^w_j\in \cA_j$ ($j\neq i$) since $M$ is indecomposable. Then we have $s(\bg^w_j)=a_1\nin \cS_i$ and $t(\bg^w_j)=a_0 \nin \cT_i$. According to  Lemma \ref{ddpoint}, this cannot happen. When $a_0$ is a source, we can similarly get a contradiction. Finally, we have
$\mid \mathcal{N}_M(a_1)\mid=n.$
\end{proof}

 When $n\geq 3$, we have the following.

\begin{proposition}\label{lj}
 Let $(a_0, a_1, \cdots, a_m)$ be a  walk of  $T(M), m\geq 2$. Suppose that any $b_i\in \mathcal{N}_M(a_1)\setminus \{a_2 \}$ is a leaf for $1 \leq i \leq n-1$. Then

\begin{enumerate}
\item  When  $a_0$ is a sink, the kernels of the maps $\bg^1_i: M_{a_1}\rightarrow M_{b_i}$  satisfy    
\begin{center}
 $\ker \bg^1_1\cong \cdots \cong  \ker \bg^1_{n-1}$, 
\end{center}
and they are all surjective. In particular,  the map $\bg^1_n: M_{a_1}\rightarrow M_{a_2}$ is injective;

\item   When $a_0$ is a source,  the images of the  maps $\bg^1_i: M_{b_i}\rightarrow M_{a_1}$  satisfy 
\begin{center}
$\im \bg^1_1\cong \cdots \cong\im \bg^{1}_{n-1}$,
\end{center}
 and they are all injective. In particular, the map $\bg^1_n: M_{a_2} \rightarrow M_{a_1}$ is surjective and $\im \bg^1_i\subseteq \im \bg^1_n$ under the isomorphisms, $1\leq i\leq n-1$.
\end{enumerate}
Moreover, there always exists $\rk \bg^1_n\geq  \rk \bg^1_i$,  where $\bg^1_n$ is the map between $M_{a_1}$ and $M_{a_2}$, $1\leq i\leq n-1$.

\end{proposition}

\begin{proof}
We consider the module $N=\pi_{\lambda}(M)$. 
If  $a_0$ is a sink, then we have the following  by Proposition \ref{position}  
$$\Gamma_M=\sum^n_{i=1}\bg_i=\begin{bmatrix}
\bg^1_1 &  &0 \\
\vdots &   & \vdots \\
\bg^1_{n-1} & & 0 \\
\bg^1_n &  &\Gamma'_M   
\end{bmatrix}.$$ 
We have $\rk \bg^1_n\geq\rk\bg^1_i$ from $\Gamma_M$. Otherwise, suppose that there exists some $j\in\{1,\cdots,n-1\}$ such that $\rk\bg^1_n<\rk\bg^1_j$. Since $s(\bg^1_n)=s(\bg^1_j)$, we have 
\begin{center}
$\rk(\bg_n+\bg_j)=\rk \begin{bmatrix}
\bg^1_j & 0\\
\bg^1_n & \Gamma_{i,j}
\end{bmatrix}>\rk \begin{bmatrix}
0 & 0\\
\bg^1_n & \Gamma_{i,j}
\end{bmatrix}\geq \rk\bg_n$,
\end{center}
where $\Gamma_{i,j}=\bigoplus_{\bg^i_n\in \cA'_n}\bg^i_n +\bigoplus_{\bg^i_j\in \cA'_j}\bg^i_j$, $\cA'_i=\cA_i\setminus \{\bg^1_i\}$, $i\in \{j,n\}$, and this is  a contradiction.  For the map $\bg^1_i: M_{a_1}\rightarrow M_{b_i}$, $b_i\in \mathcal{N}_M(a_1)\setminus \{a_2 \}$, $1\leq i\leq n-1$,  there exists an equation
 \begin{center}
  $\rk \bg^1_1= \cdots=\rk\bg^1_{n-1}=\rk\begin{bmatrix}
 \bg^1_1 \\
 \vdots\\
 \bg^1_{n-1}
 \end{bmatrix}$.
 \end{center}
Otherwise, suppose that there exists  $\rk \bg^1_i>\rk \bg^1_j$ for some  $1\leq i\neq j\leq n-1, n\geq 3$.   By Lemma \ref{matrix}, there exist inequalities

\begin{center}

  $\rk (\gamma_i+\gamma_j+\gamma_n)=\rk\begin{bmatrix} 
\bg^1_i &  &  0 \\
\bg^1_j& &  0 \\
\bg^1_n &  & \Gamma_{i,j,n} 
\end{bmatrix}\geq \rk \bg^1_i+ \rk \Gamma_{i,j,n} > \rk \begin{bmatrix} 

\bg^1_j & 0 \\
0 & \Gamma_{i,j,n}
\end{bmatrix} = \rk \bg^1_j+ \rk \Gamma_{i,j,n} \geq \rk \bg_j$
 \end{center}
where $\Gamma_{i,j,n}=(\bigoplus_{\bg^p_i\in \cA'_i}\bg^p_i)+(\bigoplus_{\bg^p_j\in \cA'_j}\bg^p_j)+(\bigoplus_{\bg^p_n\in \cA'_n}\bg^p_n), \cA'_q=\cA_q\setminus \{\bg^1_q\}$, $q\in \{i,j,n\}$,  and this is a contradiction.  Finally, we can see that  
\begin{center}
 $\ker \bg^1_1\cong \cdots \cong \ker  \bg^1_{n-1}\cong  \ker \begin{bmatrix}
 \bg^1_1 \\
 \vdots\\
 \bg^1_{n-1}
 \end{bmatrix}.$
\end{center} 
Moreover, each $\gamma^1_i$ is surjective since $M$ is indecomposable and $t(\gamma^1_i)$ is a leaf, $1\leq i \leq n-1$. Note that 
\begin{center}
 $\bg^1_i(\ker \bg^1_n)\subseteq \im \bg_n$, $1\leq i\leq n-1$
\end{center}
by  Lemma \ref{tt}. However, we have $\im \bg^1_i \cap \im \bg_n=(0)$, $1\leq i\leq n-1$. Hence  $\bg^1_i(\ker \bg^1_n)=(0)$, i.e., $\ker \bg^1_n \subseteq \ker \bg^1_i$ for all $1\leq i\leq n-1$. Then we have $\ker \bg^1_n=(0)$ because $M$ is indecomposable.

Suppose that $a_0$ is a source vertex. By duality,  we can get  the following 
\begin{center}
$\Gamma_M=\sum^n_{i=1}\bg_i=\begin{bmatrix}
\bg^1_1 & \bg^1_2 & \cdots & \bg^1_{n-1} & \bg^1_n \\
0& 0   & \cdots &  0 & \Gamma'_M  \\

\end{bmatrix}.$
\end{center}
Similarly,  we also have 
\begin{center}
$\rk \bg^1_1=\cdots= \rk \bg^1_{n-1}= \rk \begin{bmatrix}
\bg^1_1 & \cdots &\bg^1_{n-1}
\end{bmatrix}$, and  $\rk \bg^1_n\geq  \rk\bg^1_i, 1\leq i\leq n-1$.
\end{center}
Then we have $\im \bg^1_1\cong \cdots \cong \im \bg^1_{n-1}\cong\im \begin{bmatrix}
\bg^1_1 & \cdots & \bg^1_{n-1}
\end{bmatrix}$ by the theory of matrix.  Moreover, each $\bg^1_i$ is injective and $\bg^1_n$ is surjective since $M$ is indecomposable and $s(\gamma^1_i)$ is a leaf, and we clearly have $\im \bg^1_i\subseteq \im \bg^1_n$ because $\rk \Gamma_M=\rk \bg_i=\rk\bg_n$, $1\leq i \leq n-1$.

\end{proof}

We call an indecomposable regular module $M\in \modd(T(n),\Omega)$ \textit{quasi-simple} if it is located at the bottom in a regular component $\cD$, and we have 

\begin{corollary}
If $M$ is a thin module, then we have $M\in\Inj\cup\Sur$ and $M$ is quasi-simple.
\end{corollary}

\begin{proof}
Let $(a_0,\cdots,a_m)$ be  walk of $T(M)$ such that $a_0$ is a leaf, $m\geq 2$. When  $a_0$ is a sink, we consider matrix $\Gamma_M$
$$\Gamma_M=\sum^n_{i=1}\bg_i=\begin{bmatrix}
\bg^1_1 &  &0 \\
\vdots &   & \vdots \\
\bg^1_{n-1} & & 0 \\
\bg^1_n &  &\Gamma^{1}_M 
\end{bmatrix}.$$
Since $M$ is a thin module, then we have 
\begin{center}
$\ker \gamma^1_1= \cdots =\ker \gamma^1_n=(0)$ and $\rk \Gamma_M=\rk\gamma^1_1+\rk \Gamma^1_M$.

\end{center}
If we delete the set  $\cN_M(a_1)\setminus\{a_2\}$ and all maps $\bg^1_i:M_{a_1}\rightarrow M_b$ in $T(M)$, we will get a new module $M^{1}$ and a matrix $\Gamma_{M^1}$, where $b\in \cN_M(a_1)$, $1\leq i\leq n$,  It is not hard to see that $M^{1}$ is still a thin module and indecomposable. Moreover, we can see that $M^{1}$ is also of constant rank and $\rk \Gamma_{M^1}=\rk\Gamma^1_M=\rk \Gamma_M-1$. Hence we can repeat our process, and delete the set $\cN_M(a_3)\setminus\{a_4\}$ and maps $\bg^2_i:M_{a_3}\rightarrow M_b$ in $T(M^1)$, $b\in \cN_{M^1}(a_3)$. After finite steps, we will get a module $M^p$ and $\pi_\lambda(M^p)\cong P(1)$, that is to say, every leaf of $T(M)$ is a sink, that is, $M\in \Inj$. Similarly, we can show that $M\in \Sur$ when $a_0$ is a source. We get $M$ being quasi-simple by \cite[Corollary 10.16]{Daniel}.

\end{proof}

\begin{corollary}\label{dm}
If $d(M)\leq 2$,  then we have $M\in\Inj\cup\Sur$.

\end{corollary}
\begin{proof}
If $d(M)=0$, then module $\pi_\lambda(M)$ is simple we clearly have $M\in\Inj\cup\Sur$. Suppose that $d(M)=2$. Then the module is a sink or source module, that is,  
\begin{center}
$\Gamma_M=\begin{bmatrix}
\bg^1_1 & \bg^1_2 & \cdots & \bg^1_n
\end{bmatrix}$, or $\Gamma_M=\begin{bmatrix}
\bg^1_1 \\
 \bg^1_2 \\
  \vdots \\
   \bg^1_n
\end{bmatrix}$.
\end{center}
Then we have $\rk \bg^1_i=\rk \bg^1_j$ for all $1\leq i,j\leq n$. Since $M$ is indecomposable, we get $M_a=k$ for any $a\in T(M)_0$ by Lemma \ref{lj}, that is, $M\in \Inj\cup \Sur$.

\end{proof}

  Let $M\in$ mod$(T(n), \Omega)$ be an indecomposable module, $n\geq 3$, and let $M$ be of constant rank. From now on, we always assume that $d(M)\geq 3$. For any direct summand $\bg^p_q$ we  define
\begin{center}
$U^p_q= \ker \bg^p_q$ and  $\tilde{U}^p_q=\im\bg^p_q$. 
\end{center}
We let $V^p_q$ and $\tilde{V}^p_q$ be two complement vector spaces, respectively,  such that $M_{s(\bg^p_q)}=U^p_q \oplus V^p_q $ and $M_{t(\bg^p_q)}=\tilde{U}^p_q\oplus \tilde{V}^p_q $. Clearly, there exists many different pairs $(p,q)$ and $(p', q')$ such that  $U^{p}_q\cong U^{p'}_{q'}$ and $V^p_q\cong V^{p'}_{q'}$. Let $f$ be a linear map from a finite-dimensional vector space $U$ to  a finite-dimensional vector space $V$, and let $U'$ be a subspace of $U$. Then we use $f\mid_{U'}$ to denote the restriction of $f$ on $U'$, and we sometimes use $f\mid_{U'}$ to denote a trivial extended linear map from $U$ to $W$ by abusing the symbols. 

For modules of constant rank type in $\modd(T(n),\Omega)$, we have the following.

\begin{Lemma}\label{two leaves}
Let $M\in \modd(T(n),\Omega)$ be an indecomposable module. Suppose that $M$ is of constant rank.  If there exist $\bg^1_j$, $\bg^2_j\in \cA_j$ and $\bg^1_i\in \cA_i$ such that $s(\bg^1_i)=s(\bg^1_j)$, $t(\bg^1_i)=t(\bg^2_j)$ with $t(\bg^1_j)$ and $s(\bg^2_j)$ are leaves, then we have 
\begin{center}
$\rk \bg^1_i=\rk \bg^1_j+\rk\bg^2_j $ and $\im \bg^1_i\mid_{U^1_j}\cong \im \bg^2_j$.
\end{center}

\end{Lemma}

\begin{proof}
Let $\bg'_i=\bigoplus_{\gamma^j_i\in \cA'_i}\bg^j_i$ and $\bg'_j=\bigoplus_{\bg^i_j\in \cA'_j}\bg^i_j$, where $\cA'_i=\cA_i\setminus\{\bg^1_i\}$ and $\cA'_j=\cA_j\setminus\{ \bg^1_j, \bg^2_j\}$, $1\leq i\neq j\leq n$. Since    $t(\bg^1_j)$ and $s(\bg^2_j)$ are leaves,     we have the following by Lemma \ref{point}
\begin{center}
$\Gamma_M= \begin{bmatrix}
0 & \bg^1_1 & 0  & \cdots & 0 & 0 & 0\\
\vdots &\vdots & \vdots & \ddots & \vdots & \vdots & \vdots  \\
 0 &\bg^1_j &  0 &  \cdots & 0 & 0 & 0 \\
\bg^2_j &\bg^1_i & \bg^2_2 & \cdots & \bg^2_n  & 0 & 0  \\
0& 0  & \Gamma_1 & \cdots & \Gamma_{n-2} & \Gamma_{n-1} & \Gamma_{n} 
\end{bmatrix}, $
\end{center}
where $\begin{bmatrix}
0 & 0 &\Gamma_1&   \cdots & \Gamma_n  
\end{bmatrix}= \sum^n_{w=1,w\neq i,j} \bg'_w+\bg'_i+\bg'_j$, and $\bg'_w=\bigoplus_{\bg^j_w\in (\cA_w\setminus\{\bg^2_w\})}\bg^j_w$. Note that $\bg^1_j$ is surjective and $\bg^2_j$ is injective by Lemma \ref{lj}. Using $\rk \Gamma_M=\rk \bg_i=\rk(\bg_i+ \bg_j)$, we have $\rk \bg^1_i=\rk \bg^1_j+\rk\bg^2_j$. According to Lemma \ref{tt}, there exists $ \bg^1_i(\ker \bg^1_j)\subseteq \im \bg^2_j$. If $\bg^1_i(\ker \bg^1_j)\subsetneq \im \bg^2_j$, then $\rk \bg^1_i\leq  \rk \bg^1_i\mid_{U^1_j}+\rk \bg^1_i\mid_{V^1_j}<\rk \bg^2_j+\rk \bg^1_j\mid_{V^1_j}=\rk \bg^2_j+\rk\bg^1_j$ because $\im \bg^1_j=\im \bg^1_j\mid_{V^1_j}$, and this is a contradiction. Hence we have $\im \bg^1_i\mid_{U^1_j}\cong \im \bg^2_j$. 

\end{proof}

Following the conditions in Lemma \ref{two leaves},   if we choose two bases of vector spaces $M_{s(\bg^1_i)}\oplus M_{s(\bg^2_j)} $ and $M_{t(\bg^1_i)}\oplus M_{t(\bg^1_j)}$, respectively, then we have the following

\begin{center}
 $ \begin{bmatrix}
\bg^1_j & 0\\
\bg^1_i & \bg^2_j
\end{bmatrix}=\begin{bmatrix}
b_{11} & b_{12} & \cdots & b_{1u} & 0 &\cdots & 0 \\
b_{21} &  b_{22} & \cdots & b_{2u} & 0 &  \cdots & 0 \\
\vdots & \vdots & \ddots & \vdots & 0 &  \ddots & 0 \\
b_{x1} & b_{x2} & \cdots &b_{xu} & 0 & \cdots & 0 \\
a_{11} & a_{12} & \cdots & a_{1u} & 0 & \cdots  & 0  \\
\vdots & \vdots & \ddots & \vdots  & \vdots & \ddots & \vdots \\
a_{p1} &  a_{p2} & \cdots & a_{pu} &0 & \cdots & 0 \\
a_{p+1 1} &  a_{p+1 2} & \cdots & a_{p+1 u} & c_{11} & \cdots & c_{1v} \\
\vdots & \vdots & \ddots & \vdots  & \vdots & \ddots & \vdots \\
a_{z1} & a_{z2} & \cdots &a_{zu}  & c_{y1} & c_{y2} & \cdots c_{yv} \end{bmatrix}$,

\end{center}
where $\bg^1_j=(b_{ij})_{x\times u}$, $\bg^1_i=(a_{ij})_{z\times u}$, $\bg^2_j=(c_{ij})_{y\times v}$, \hspace{0.1cm} $x=\dim_k M_{t(\bg^1_j)}$, $y=\dim_k \im \bg^2_j$, and $z=\dim_k M_{t(\bg^1_i)}$, $u=\dim_k M_{s(\bg^1_i)}$, $v=\dim_k M_{s(\bg^2_j)}$. Moreover, we have
\begin{center}
$\rk \bg^1_j=\rk ((a_{ij})_{p\times u})=x$ and $\rk \bg^2_j=\rk ((a_{ij})_{qu})=y$,
\end{center}
where $p\leq q \leq z$.  Based on this observation,  we finally have the following conclusion.

\begin{Theorem}\label{es}
Let $M\in \modd(T(n),\Omega)$ be an indecomposable module, $n\geq 3$. Suppose that $M$ is of constant rank. Then

\begin{center}
$M\in \Inj\cup \Sur$.
\end{center}

\end{Theorem}
\begin{proof}
We  assume that $M$ is a sink or flow module since the case of source module is similar.  We now use the induction on the diameter $d(M)$ of $M$.  If $d(M)\leq 2$,  then we already know that $M\in \Inj\cup \Sur$ by Corollary  \ref{dm}.
 
Suppose that $d(M)=3$. According to  Proposition \ref{lj},  We have 

\begin{equation}\label{og}
\Gamma_M=\sum^n_i \bg_i=\begin{bmatrix}
\bg^1_1 & 0 & 0&  \cdots & 0 \\
\bg^1_2 & 0& 0&  \cdots & 0  \\
\vdots & \vdots & \vdots  & \ddots  & \vdots \\
\bg^1_n & \bg^2_1 & \bg^2_2 & \cdots & \bg^2_{n-1}
\end{bmatrix}. 
\end{equation}
Then  we get  
\begin{equation}\label{m1}
 \rk \bg^1_n=\rk\bg^1_i+\rk \bg^2_i  
\end{equation}
by Lemma \ref{two leaves},   $1\leq i\leq n-1$. Moreover, we have

\begin{equation}\label{Step1}
\im \bg^2_i \cong \bg^1_n\mid_{U^1_1} \text{ and } \rk \begin{bmatrix}
\bg^1_1\mid_{V^1_1} \\
\bg^1_2\mid_{V^1_1}\\
\vdots\\
\bg^1_n\mid_{V^1_1}
\end{bmatrix}=\rk \bg^1_i , 1\leq i \leq n-1.
\end{equation}
If we choose two bases of $\pi_\lambda(M)_1$ and $\pi_\lambda(M)_2$ , then we will have 

\begin{equation}
\Gamma_M= \begin{bmatrix}
\bg^1_1\mid_{V^1_1} \\
\bg^1_2\mid_{V^1_1} \\
\vdots \\
\bg^1_{n}\mid_{V^1_1} 
\end{bmatrix} \bigoplus \begin{bmatrix}
\bg^1_n\mid_{U^1_1} & \bg^2_1 & \bg^2_2 & \cdots & \bg^2_{n-1}
\end{bmatrix},
\end{equation}
under the isomorphisms.   Then $M$ is decomposable, and this is a contradiction. Suppose that $d(M)=4$ and $M$ is a sink module.  By Lemma \ref{point},  we get  $\mid\cN_M(s(\bg^2_i))\mid\in \{1,n\}$ i.e., either $s(\bg^2_i)$ is a leaf or has $n$ neighbours comparing  with the matrix in $(\ref{og})$,  $1\leq i\leq n-1$.  Then the biggest size matrix $\Gamma_M$ will look like the following
\begin{equation}\label{bmatrix}
\Gamma_M=\begin{bmatrix}
\bg^1_1 & 0  & 0 & \cdots & 0 \\
\bg^1_2 & 0 & 0& \cdots & 0\\
\vdots &   \vdots & \vdots & \ddots  & \vdots \\
\bg^1_n & \bg^2_1 & \bg^2_2& \cdots & \bg^2_{n-1}\\
0 & \bg^3_2 & 0 & \cdots & 0 \\
0 & \bg^3_3 & 0 & \cdots & 0 \\
\vdots &   \vdots & \vdots & \ddots  & \vdots \\
0 & \bg^3_n & 0 & \cdots & 0 \\
0 & 0 & \bg^4_1 & \cdots & 0 \\
0 & 0 & \bg^4_3 & \cdots & 0 \\
\vdots &   \vdots & \vdots & \ddots  & \vdots \\
0 & 0 & \bg^4_n & \cdots & 0 \\
0 & 0 & 0 & \cdots & 0 \\
0 & 0 & 0 & \cdots & 0 \\
\vdots &   \vdots & \vdots & \ddots  & \vdots \\
0 & 0 & 0 & \cdots & 0 \\
0 & 0 & 0 & \cdots & \bg^{n+1}_1 \\
0 & 0 & 0 & \cdots & \bg^{n+1}_2 \\
\vdots &   \vdots & \vdots & \ddots  & \vdots \\
0 & 0 & 0 & \cdots & \bg^{n+1}_n \\
\end{bmatrix}.
\end{equation}
We take it as a special case of  $\Gamma_M$ when $\mid\cN_M(s(\bg^2_i))\mid=1$ for some $i\in \{1,2,\cdots,n-1\}$. Note that $\bg^1_n$ and $\bg^2_i$ are injective by Proposition \ref{lj} $(a)$, $1\leq i\leq n-1$. Since $s(\bg^i_1)=s(\bg^i_n)$,  $t(\bg^i_1)$ and $t(\bg^i_n)$ are all leaves for all $4 \leq i \leq n+1$, we have 

\begin{equation}\label{r1}
\rk \begin{bmatrix}
\bg^i_1\\
\bg^i_n
\end{bmatrix}=\rk \bg^i_1=\rk \bg^i_n, 4\leq i\leq n+1.
\end{equation}
Hence we have 
\begin{equation}\label{r2}
\rk \begin{bmatrix}
\bg^1_1& 0\\
\bg^1_n & \bg^2_1 \\
0 & \bg^3_n
\end{bmatrix}=\rk \bg^1_1+\rk \bg^2_1=\rk \bg^1_n+\rk \bg^3_n .
\end{equation}
We decomposable it into  the following form

\begin{center}
$ \begin{bmatrix}
\bg^1_1\mid_{V^1_1} & 0 \\
\bg^1_n\mid_{V^1_1}+ \bg^1_n\mid_{U^1_1} & \bg^2_1\mid_{V^3_n}+\bg^2_1\mid_{U^3_n}\\
0 & \bg^3_n\mid_{V^3_n}
\end{bmatrix}. $
\end{center}
Since $\bg^1_n$ and $\bg^2_1$ are injective, the maps $\bg^1_n\mid_{V^1_1}$ and $\bg^2_1\mid_{V^3_n}$ are also injective. Clearly, we have 
\begin{center}
$\rk \begin{bmatrix}
\bg^1_1\mid_{V^1_1}\\
\bg^1_n\mid_{V^1_1}
\end{bmatrix}=\rk \bg^1_1=\rk \bg^1_n\mid_{V^1_1}$.
\end{center}
Similarly, we have  $\rk \begin{bmatrix}
\bg^2_1\mid_{V^3_n} \\
\bg^3_n\mid_{V^3_n}
\end{bmatrix}=\rk \bg^3_n=\rk \bg^2_1\mid_{V^3_n}$. In the end, we can see that  
\begin{equation}\label{r3}
\im \bg^1_n\mid_{U^1_1}\cong \im \bg^2_1\mid_{U^3_2}.
\end{equation}
Basically, it says that we can replace $\bg^2_1$ with $\bg^2_1\mid_{U^{3}_n}$ when $i=1$ in the equation (\ref{Step1}). For the rest $\bg^2_i$ we can do the same thing, $2\leq i\leq n-1$, and  we finally  get the following  $\Gamma_M$ if we take two vector spaces as the same if they are isomorphic
\begin{equation}\label{cyc}
\Gamma_M=\begin{bmatrix}
\bg^1_1\mid_{V^1_1} & 0  & 0 & \cdots & 0 \\
\bg^1_2\mid_{V^1_1} & 0 & 0& \cdots & 0\\
\vdots &   \vdots & \vdots & \ddots  & \vdots \\
\bg^1_n\mid_{V^1_1} & \bg^2_1\mid_{V^3_2} & \bg^2_2\mid_{V^4_1}& \cdots & \bg^2_{n-1}\mid_{V^{n+1}_1}\\
0 & \bg^3_2\mid_{V^3_2} & 0 & \cdots & 0 \\
0 & \bg^3_3\mid_{V^3_2} & 0 & \cdots & 0 \\
\vdots &   \vdots & \vdots & \ddots  & \vdots \\
0 & \bg^3_n\mid_{V^3_2} & 0 & \cdots & 0 \\
0 & 0 & \bg^4_1\mid_{V^4_1} & \cdots & 0 \\
0 & 0 & \bg^4_3\mid_{V^4_1} & \cdots & 0 \\
\vdots &   \vdots & \vdots & \ddots  & \vdots \\
0 & 0 & \bg^4_n\mid_{V^4_1} & \cdots & 0 \\
0 & 0 & 0 & \cdots & 0 \\
0 & 0 & 0 & \cdots & 0 \\
\vdots &   \vdots & \vdots & \ddots  & \vdots \\
0 & 0 & 0 & \cdots & 0 \\
0 & 0 & 0 & \cdots & \bg^{n+1}_1\mid_{V^{n+1}_1} \\
0 & 0 & 0 & \cdots & \bg^{n+1}_2\mid_{V^{n+1}_1} \\
\vdots &   \vdots & \vdots & \ddots  & \vdots \\
0 & 0 & 0 & \cdots & \bg^{n+1}_n\mid_{V^{n+1}_1} \\
\end{bmatrix} \bigoplus \begin{bmatrix}
\bg^1_n\mid_{U^1_1} & \bg^2_1\mid_{U^3_2} &\bg^2_2\mid_{U^4_1} & \cdots & \bg^2_{n-1}\mid_{U^{n+1}_1}
\end{bmatrix}.
\end{equation}
Hence we can see that $M$ is decomposable if  $U^j_i\neq 0$ for some $i$ and $j$, $i\in \{1,2,\cdots,n\}$, $j\in \{1,2\}$.   
We continue doing it by discussing the diameter $d(M)$,  and by replacing the maps in equations $(\ref{og})-(\ref{cyc})$ if there exists some $\bg^i_n$ (or $\bg^1_j$) such that $t(\bg^i_n)$ ( or $t(\bg^1_j)$) is not a leaf,  $ 3\leq i\leq n+1$, $1\leq j\leq n-1$. For example, if $d(M)=5$ and $t(\bg^3_n)$ is not a leaf in matrix ($\ref{bmatrix}$) such that

\begin{equation}\label{five}
\Gamma_M=\begin{bmatrix}
\bg^1_1 & 0  & 0 & \cdots & 0  & 0 & 0& \cdots & 0\\
\bg^1_2 & 0 & 0& \cdots & 0 & 0 & 0& \cdots & 0\\
\vdots &   \vdots & \vdots & \ddots  & \vdots  & \vdots & \vdots& \ddots & \vdots\\
\bg^1_n & \bg^2_1 & \bg^2_2& \cdots & \bg^2_{n-1} & 0 & 0& \cdots & 0\\
0 & \bg^3_2 & 0 & \cdots & 0  & 0 & 0& \cdots & 0\\
0 & \bg^3_3 & 0 & \cdots & 0  & 0 & 0& \cdots & 0\\
\vdots &   \vdots & \vdots & \ddots  & \vdots   & \vdots & \vdots& \ddots & \vdots\\
0 & \bg^3_n & 0 & \cdots & 0& \bg^{n+2}_1 & \bg^{n+2}_2 & \cdots & \bg^{n+2}_{n-1} \\
0 & 0 & \bg^4_1 & \cdots & 0  & 0 & 0& \cdots & 0\\
0 & 0 & \bg^4_3 & \cdots & 0  & 0 & 0& \cdots & 0\\
\vdots &   \vdots & \vdots & \ddots  & \vdots   & \vdots & \vdots& \ddots & \vdots\\
0 & 0 & \bg^4_n & \cdots & 0  & 0 & 0& \cdots & 0\\
0 & 0 & 0 & \cdots & 0  & 0 & 0& \cdots  & 0\\
0 & 0 & 0 & \cdots & 0  & 0 & 0& \cdots  & 0\\
\vdots &   \vdots & \vdots & \ddots  & \vdots  & \vdots & \vdots&\ddots & \vdots \\
0 & 0 & 0 & \cdots & 0  & 0 & 0& \cdots  & 0\\
0 & 0 & 0 & \cdots & \bg^{n+1}_1  & 0 & 0& \cdots  & 0\\
0 & 0 & 0 & \cdots & \bg^{n+1}_2  & 0 & 0& \cdots  & 0\\
\vdots &   \vdots & \vdots & \ddots  & \vdots   & \vdots & \vdots& \ddots & \vdots\\
0 & 0 & 0 & \cdots & \bg^{n+1}_n  & 0 & 0& \cdots & 0\\
\end{bmatrix}.
\end{equation}
where $s(\bg^{n+2}_i)$ is a leaf, $1\leq i\leq n-1$. We compare this matrix with the block matrix  $(\ref{bmatrix})$ when $d(M)=4$. By Proposition \ref{lj} $(b)$, we know that $\bg^{n+4}_i$ is injective, $\bg^3_n$ is surjective and $\im \bg^{n+2}_i\subseteq \im \bg^3_n$ under the isomorphisms.  Then 

\begin{center}
$ \im \bg^3_n\mid_{U^3_i}\cong \im \bg^{n+2}_i\cong \im \bg^{n+2}_1$ and $\rk \bg^3_n=\rk \bg^3_i+\rk \bg^{n+2}_i$, $2\leq i\leq n-1$,
\end{center}
and we have

\begin{equation}
\rk \begin{bmatrix}
\bg^1_1 & 0 & 0 \\
\bg^1_n\mid_{V^1_1}+\bg^1_n\mid_{U^1_1} & \bg^2_1\mid_{V^3_2}+\bg^2_1\mid_{U^3_2} & 0\\
0 & \bg^3_n& \bg^{n+2}_1
\end{bmatrix}=\rk \begin{bmatrix}
\bg^1_1\mid_{V^1_1} \\
\bg^1_n\mid_{V^1_1}
\end{bmatrix} +\rk  \begin{bmatrix}
0& 0 & 0 \\
\bg^1_n\mid_{U^1_1} &\bg^2_1\mid_{V^3_2}+\bg^2_1\mid_{U^3_2} & 0\\
0 & \bg^3_n\mid_{V^3_2}+\bg^3_n\mid_{U^3_2} & \bg^{n+2}_1
\end{bmatrix}.   
\end{equation}
It says that we have the matrix

\begin{equation*}
\Gamma_M=\begin{bmatrix}
\bg^1_1\mid_{V^1_1} & 0  & 0 & \cdots & 0 \\
\bg^1_2\mid_{V^1_1} & 0 & 0& \cdots & 0\\
\vdots &   \vdots & \vdots & \ddots  & \vdots \\
\bg^1_n\mid_{V^1_1} & \bg^2_1\mid_{V^3_2} & \bg^2_2\mid_{V^4_1}& \cdots & \bg^2_{n-1}\mid_{V^{n+1}_1}\\
0 & \bg^3_2\mid_{V^3_2} & 0 & \cdots & 0 \\
0 & \bg^3_3\mid_{V^3_2} & 0 & \cdots & 0 \\
\vdots &   \vdots & \vdots & \ddots  & \vdots \\
0 & \bg^3_n\mid_{V^3_2} & 0 & \cdots & 0 \\
0 & 0 & \bg^4_1\mid_{V^4_1} & \cdots & 0 \\
0 & 0 & \bg^4_3\mid_{V^4_1} & \cdots & 0 \\
\vdots &   \vdots & \vdots & \ddots  & \vdots \\
0 & 0 & \bg^4_n\mid_{V^4_1} & \cdots & 0 \\
0 & 0 & 0 & \cdots & 0 \\
0 & 0 & 0 & \cdots & 0 \\
\vdots &   \vdots & \vdots & \ddots  & \vdots \\
0 & 0 & 0 & \cdots & 0 \\
0 & 0 & 0 & \cdots & \bg^{n+1}_1\mid_{V^{n+1}_1} \\
0 & 0 & 0 & \cdots & \bg^{n+1}_2\mid_{V^{n+1}_1} \\
\vdots &   \vdots & \vdots & \ddots  & \vdots \\
0 & 0 & 0 & \cdots & \bg^{n+1}_n\mid_{V^{n+1}_1} \\
\end{bmatrix} \bigoplus \begin{bmatrix}
\bg^1_n\mid_{U^1_1} & \bg^2_1\mid_{U^3_2} &\bg^2_2\mid_{U^4_1} & \cdots & \bg^2_{n-1}\mid_{U^{n+1}_1} & 0 & \cdots & 0\\
0 & \bg^3_n\mid_{U^3_2} & 0 & \cdots & 0 & \bg^{n+2}_1 & \cdots & \bg^{n+2}_{n-1}
\end{bmatrix}.
\end{equation*}
Module $M$ is also decomposable.

We repeat the above process by inducting the length of diameter  walk of $T(M)$, and analysing  the block matrices $\Gamma_M$ or its transpose $\Gamma^t_M$:  we  write each map $\bg^i_j$ into the form 
 $\bg^i_j=\bg^i_j\mid_{V^p_q}+\bg^i_j\mid_{U^p_q}$ with some map $\bg^p_q$ when $s(\bg^i_j)=s(\bg^p_q)$ and $t(\bg^p_q)$ is an ending point in a diameter walk of $T(M)$. 
 According to Proposition \ref{position}, we have  $\rk \Gamma_M= \rk \bg^1_1\mid_{V^1_1}+\rk \Gamma'_M$ when $M$ is of constant rank. Note that the rank of $\Gamma'_M$ is also fixed.  Hence we can replace $\Gamma'_M$ with $\Gamma_M$ and continue our discussion. Finally, we  add all $\bg^i_j\mid_{V^p_q}$ (or $\bg^i_j\mid_{U^p_q}$) in one block matrix, and we get $\Gamma_M$ being decomposable.
 
  We can do the same thing when $M$ is a flow or source module. Finally, we have $M\in \Inj \cup \Sur$ if $M$ is indecomposable and of constant rank.

\end{proof}

We now discuss the modules of constant rank in the   regular component. Let $\cC$ be a regular component of $\modd\cK_n$. It is shown that there are uniquely determined quasi-simple modules $M_{\cC}$ and $W_{\cC}$ in $\cC$ such that the cone $(M_{\cC}\rightarrow)$ of all successors of $M_{\cC}$ satisfies $(M_{\cC}\rightarrow)= \EKP\cap \cC$ and the cone $(\rightarrow W_{\cC})$ of all predecessors of $W_{\cC}$ satisfies $(\rightarrow W_{\cC})=\EIP \cap \cC$ \cite[Theorem 3.3]{Julia}. Then we have

\begin{Definition}
Let $\cC$ be a regular component of mod $\mathcal{K}_n$. We  define  an integer $\cW(\cC)$  satisfying $\tau^{\cW(\cC)+1}M_{\cC}=W_{\cC}$.  We call $\cW(\cC)$  the \textit{width} of $\cC$. Let $\cD$ be a regular component of mod$(T(n), \Omega)$. We similarly have an integer $\cW(\cD)$  satisfying $\tau^{\cW(\cD)+1}M_{\cD}=W_{\cD}$. We call $\cW(\cD)$  the \textit{width} of $\cD$.

\end{Definition}

By Theorem \ref{es}, we have an important conclusion.
\begin{proposition}
Let $\cD$ be a regular component of $\modd(T(n),\Omega)$. Then the modules of constant rank in $\cD$ satisfy 

\begin{center}
$\cD\bigcap \GR=\cD\bigcap (\Inj\bigcup \Sur)$.
\end{center}

\end{proposition}

This proposition tells us that there is no modules of constant rank in the middle (outside of $\Inj\cup \Sur$) of the regular components of $\modd(T(n),\Omega)$ \cite[Theorem 3.7]{Jie2}. We now give an example.

\begin{example}
Let $M\in\modd(T(3),\Omega)$ be the following 
\begin{center}

\begin{tikzpicture}
\node (00) at (0,0) {$0$};
\node (10) at (1,0) {$k$};
\node (20) at (2,0) {$k$};
\node (30) at (3,0) {$k$};
\node (40) at (4,0) {$k$};
\node (50) at (5,0) {$k$};
\node (60) at (6,0) {$0$,};
\node (1-1) at (1,-1) {$0$};
\node (2-1) at (2,-1) {$k$};
\node (3-1) at (3,-1) {$0$};
\node (4-1) at (4,-1) {$k$};
\node (5-1) at (5,-1) {$0$};

\path [->] (00) edge (10)
           (20) edge node[above]{$\gamma^1_1$}(10) 
           (1-1) edge (10)
          (20) edge node[above]{$\gamma^1_3$} (30)
            (40) edge node[above]{$\gamma^2_1$} (30)               
             (20) edge node[midway, right]{$\gamma^1_2$} (2-1)            
                (40) edge node[above]{$\gamma^2_3$} (50)
                (40) edge node[midway, right]{$\gamma^2_2$} (4-1)
                (60) edge (50)
                (5-1) edge (50)
              (3-1) edge (30)
                      
                           ;
\end{tikzpicture}
\end{center}
where $\gamma^i_j\in k\setminus\{0\}$ and $\pi(\gamma^i_j)=\gamma_j$, $i\in\{1,2\}$, $j\in \{1,2,3\}$. Then $M$ is regular.  By direct computation, we have $M\in \Inj$ and $\tau^2 M\in \Sur$ (see Figure 2).

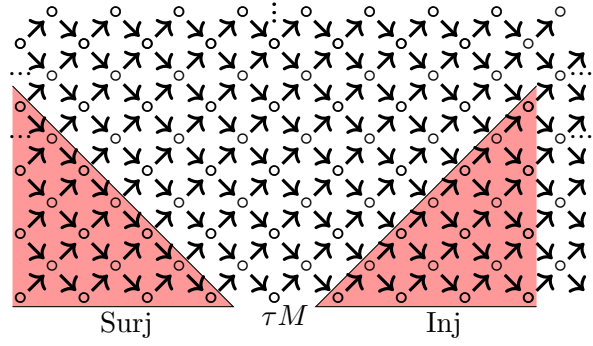
\begin{figure}[!h]

\begin{center}

\begin{tikzpicture}[very thick,scale=0.7]

                    [every node/.style={fill, circle, inner sep = 1pt}]

%%%%%%%%%%%%%%%%%%%%%%%%%%%%%% Parameter %%%%%%%%%%%%%%%%%%%%%%%%%%%%%%%%%%%

\def \n {8} % #Knoten Reihe  - 1

\def \m {4} % #Knoten Spalte - 1

\def \translation {0} % 1 Für Translation

\def \ab {0.15} % Abstand Pfeil und Knoten

\def \Pab {0.6} % Halber Abstand Horizontal

\def \lcone {1} % 1 für linken Kegel

\def \ldist {3} % Anzahl der quasi-einfachen die eingeschlossen werden - 1

\def \lcolor {red} % white für keine Farbe

\def \rcone {1} % 1 für rechten Kegel

\def \rdist {3} % Anzahl der quasi-einfachen die eingeschlossen werden - 1

\def \rcolor {red} %  white für keine Farbe

\def \llcone {0} % 1 für einen zweiten linken Kegel rechts von lcone

\def \lldist {4} % Anzahl der quasi-einfachen die eingeschlossen werden - 1

\def \rrcone {0} %1 für einen zweiten rechten Kegel links von rcone

\def \rrdist {4} % Anzahl der quasi-einfachen die eingeschlossen werden - 1

%%%%%%%%%%%%%%%%%%%%%%%%%%%% Quellcode %%%%%%%%%%%%%%%%%%%%%%%%%%%%%%%%%%%%%%

\foreach \a in {0,...,\n}{

\foreach \b in {0,...,\m}{

    \ifthenelse{\a = \n \and \b < \m}{

   %\node[color=black] ({\a,\b}) at ({\a*2*\Pab+\Pab},{\b*2*\Pab+\Pab}) {$\circ$};

   \node[color=black] ({\a,\b,5})at ({\a*2*\Pab},{\b*2*\Pab}) {$\circ$};

     }

     {

      \ifthenelse{\b = \m \and \a < \n}{

      \node[color=black] ({\a,\b}) at ({\a*2*\Pab+\Pab},{\b*2*\Pab+\Pab}) {$\circ$};

      \node[color=black] ({\a,\b,5})at ({\a*2*\Pab},{\b*2*\Pab}) {$\circ$};

      }

      {

      \ifthenelse{\b = \m \and \a = \n}

     {\node[color=black] ({\a,\b,5})at ({\a*2*\Pab},{\b*2*\Pab}) {$\circ$};}

    {\node[color=black] ({\a,\b}) at ({\a*2*\Pab+\Pab},{\b*2*\Pab+\Pab}) {$\circ$};

    \node[color=black] ({\a,\b,5})at ({\a*2*\Pab},{\b*2*\Pab}) {$\circ$};

      }

      }

      }

    }

    }

\foreach \s in {0,...,\n}{

\foreach \t in {0,...,\m}{  

 \ifthenelse{\t = \m \and \s < \n}{

    \draw[->] (\s*2*\Pab+\ab,\t*2*\Pab+\ab) to (\s*2*\Pab+\Pab-\ab,\t*2*\Pab+\Pab-\ab); 

    \draw[->] (\s*2*\Pab+\Pab+\ab,\t*2*\Pab+\Pab-\ab) to (\s*2*\Pab+2*\Pab-\ab,\t*2*\Pab+\ab); 

 }{

   \ifthenelse{\s = \n \and \t < \m}{

    %\draw[->] (\s*2*\Pab+\Pab+\ab,\t*2*\Pab+\Pab+\ab) to (\s*2*\Pab+2*\Pab-\ab,\t*2*\Pab+2*\Pab-\ab);

   %\draw[->] (\s*2*\Pab+\ab,\t*2*\Pab+2*\Pab-\ab) to (\s*2*\Pab+\Pab-\ab,\t*2*\Pab+\Pab+\ab);  

   %\draw[->] (\s*2*\Pab+\ab,\t*2*\Pab+\ab) to (\s*2*\Pab+\Pab-\ab,\t*2*\Pab+\Pab-\ab); 

   }

  {

  \ifthenelse{\s = \n \and \t = \m}{

    }{

   \draw[->] (\s*2*\Pab+\ab,\t*2*\Pab+\ab) to (\s*2*\Pab+\Pab-\ab,\t*2*\Pab+\Pab-\ab); 

   \draw[->] (\s*2*\Pab+\Pab+\ab,\t*2*\Pab+\Pab+\ab) to (\s*2*\Pab+2*\Pab-\ab,\t*2*\Pab+2*\Pab-\ab);

   \draw[->] (\s*2*\Pab+\ab,\t*2*\Pab+2*\Pab-\ab) to (\s*2*\Pab+\Pab-\ab,\t*2*\Pab+\Pab+\ab); 

   \draw[->] (\s*2*\Pab+\Pab+\ab,\t*2*\Pab+\Pab-\ab) to (\s*2*\Pab+2*\Pab-\ab,\t*2*\Pab+\ab);    

   }

   }

   }

    }

    }

    \draw[->] (\n*2*\Pab+\ab,\m*\Pab+2*\Pab+\Pab+\Pab+\ab) to (\n*2*\Pab+\Pab-\ab,\m*\Pab+2*\Pab+\Pab+\Pab+\Pab-\ab); 

\ifthenelse{\isodd{\m}}

%% IF

 { 
 
  \node[color=black] (Dots1) at (0,\m*\Pab+2*\Pab+\Pab) {$\cdots$};

  \node[color=black] (Dots2) at (1+\n*2*\Pab,\m*\Pab+2*\Pab+\Pab) {$\cdots$};

   \ifthenelse{\isodd{\n}}{

  \node[color=black] (Dots3) at (0.5*\n*2*\Pab,2*\m*\Pab+2*\Pab) {$\vdots$};}

  {\node[color=black] (Dots3) at (0.5*\n*2*\Pab,2*\m*\Pab+\Pab) {$\vdots$};} 

  }

%% Else

  {

  \node[color=black] (Dots1) at (0,\m*\Pab+\Pab) {$\cdots$};

  \node[color=black] (Dots2) at (1+\n*2*\Pab,\m*\Pab+\Pab) {$\cdots$};

  \ifthenelse{\isodd{\n}}{

  \node[color=black] (Dots3) at (0.5*\n*2*\Pab,2*\m*\Pab+2*\Pab) {$\vdots$};}

  {\node[color=black] (Dots3) at (0.5*\n*2*\Pab,2*\m*\Pab+\Pab) {$\vdots$};}

  }

\ifthenelse{\translation = 1}{

   \foreach \s in {0,...,\n}{

   \foreach \t in {0,...,\m}{ 

   \ifthenelse{\s = 0}{}{

      \ifthenelse{\s = \n}{\draw[->,dotted,thin] (\s*2*\Pab-\ab,\t*2*\Pab) to (\s*2*\Pab-2*\Pab+\ab,\t*2*\Pab); }{

   \draw[->,dotted,thin] (\s*2*\Pab-\ab,\t*2*\Pab) to (\s*2*\Pab-2*\Pab+\ab,\t*2*\Pab); 

   \draw[->,dotted,thin] (\s*2*\Pab-\ab+\Pab,\t*2*\Pab+\Pab) to (\s*2*\Pab-2*\Pab+\Pab+\ab,\t*2*\Pab+\Pab); 

   }

   }}

}}

{}  %ELSE

\begin{scope}[on background layer]

 \ifthenelse{\llcone = 1}{

        \draw[fill = \lcolor!5]

        (-\ab,-\ab) node[anchor=north]{}

  -- (\lldist*\Pab*2+0.7*\Pab,-\ab) node[anchor=north]{}

  -- (-\ab,\lldist*\Pab*2+0.7*\Pab) node[anchor=south]{}; 

   }

   {}

\ifthenelse{\rrcone = 1}{

        \draw[fill = \rcolor!5] (\n*\Pab*2+\ab,-\ab) node[anchor=north]{}

  -- (\n*\Pab*2-\rrdist*\Pab*2-0.7*\Pab,-\ab) node[anchor=north]{}

  -- (\n*\Pab*2+\ab,\rrdist*\Pab*2+0.7*\Pab) node[anchor=south]{};

    }

  {}

 \ifthenelse{\llcone = 1}{

        \draw

        (-\ab,-\ab) node[anchor=north]{}

  -- (\lldist*\Pab*2+0.7*\Pab,-\ab) node[anchor=north]{}

  -- (-\ab,\lldist*\Pab*2+0.7*\Pab) node[anchor=south]{}; 

   }

   {}

\end{scope}

\begin{scope}[on background layer]

\ifthenelse{\lcone = 1}{

        \draw[fill= \lcolor!40] (-\ab,-\ab) node[anchor=north]{}

  -- (\ldist*\Pab*2+0.7*\Pab,-\ab) node[anchor=north]{}

  -- (-\ab,\ldist*\Pab*2+0.7*\Pab) node[anchor=south]{}; 

   }

   {}

\ifthenelse{\rcone = 1}{

        \draw[fill= \rcolor!40](\n*\Pab*2+\ab,-\ab) node[anchor=north]{}

  -- (\n*\Pab*2-\rdist*\Pab*2-0.7*\Pab,-\ab) node[anchor=north]{}

  -- (\n*\Pab*2+\ab,\rdist*\Pab*2+0.7*\Pab) node[anchor=south]{};

    }

  {}

\end{scope}

\node[color=black]  at (5,-0.31) {$\tau M$};   
\node[color=black]  at (8,-0.5) {$\Inj$};    
\node[color=black]  at (2,-0.5) {$\Sur$}; 
\end{tikzpicture}

\end{center}

\caption{$\cW(\cD)=1$ of regular component $\cD$}

\label{Fig:RCGM1}

\end{figure}
\end{example}

\clearpage
%%%%%%%%%%%%%%%%%%%%%%% REFERENCES %%%%%%%%%%%%%%%%%%%


\begin{bibdiv}
\begin{biblist}
\addcontentsline{toc}{chapter}{\textbf{Bibliography}}

\bib{Assem1}{book}{
title={Elements of the representation Theory of Associative Algebras, I},
subtitle={Techniques of Representation Theory},
series={London Mathematical Society Student Texts},
author={I. Assem},
author={D. Simson},
author={A. Skowro\'nski},
publisher={Cambridge University Press},
date={2006},
address={Cambridge},
}

\bib{Assem2}{book}{
title={Elements of the representation Theory of Associative Algebras, \Romannum{3}},
subtitle={Representation-infinite Tilted Algebras},
series={London Mathematical Society Student Texts},
author={I. Assem},
author={D. Simson},
author={A. Skowro\'nski},
publisher={Cambridge University Press},
date={2007},
address={Cambridge},
}

\bib{Kerner}{article}{
title={Representations of Wild Quivers Representation theory of algebras and related topics},
author={Kerner, O.},
journal={CMS
Conf. Proc.},
volume={19},
date={1996},
pages={65-107},

}

\bib{Julia}{article}{
title = {Categories of modules for elementary abelian p-groups and generalized Beilinson algebras},
author = {Julia Worch},
journal = {J.London Math.Soc},
date = {2013},
volume = {88},
pages = {649-688},
}




\bib{Carlson}{article}{ 
    author={Carlson, Jon F.},
   author={Friedlander, Eric M.},
   author={Pevtsova, Julia},
     TITLE = {Modules of constant {J}ordan type},
   JOURNAL = {J. Reine Angew. Math.},
  FJOURNAL = {Journal f\"ur die Reine und Angewandte Mathematik. [Crelle's
              Journal]},
    VOLUME = {614},
      YEAR = {2008},
     PAGES = {191--234},
      ISSN = {0075-4102,1435-5345},
   MRCLASS = {20G05 (14L15 16G70)},
  MRNUMBER = {2376286},
MRREVIEWER = {Alan\ Koch},
       DOI = {10.1515/CRELLE.2008.006},
       URL = {https://doi.org/10.1515/CRELLE.2008.006},
}




\bib{Bongartz}{article}{
title={Covering spaces in representation theory},
author={K.Bongartz and P.Gabriel},
journal={Inventiones mathematicae},
volume={65},
date={1981/82},

pages={331-378},
}


\bib{Claus}{article}{
title = {The shift orbits of the graded Kronecker modules},
author = {Claus Michael Ringel},

journal = {Mathematische Zeitschrift},
volume = {290},
date = {2018},
pages = {1199-1222},
number = {3},
}






\bib{Claus5}{webpage}{
title = {Simple representations, thin representations.},
author = {Michael Ringel, Claus},
date = {2012},
url = {https://www.math.uni-bielefeld.de/~sek/kau/leit2v2.pdf},

}

\bib{Daniel}{thesis}{
title = {Representations of Regular Trees and Invariants of AR-Components for Generalized Kronecker Quivers},
author = {Bissinger, Daniel},
school={Mathematisch-Naturwissenschaftliche
Fakultät, Christian-Albrechts-Universität zu Kiel},
type={PhD thesis},
date = {2018},
url = {https://macau.uni-kiel.de/servlets/MCRFileNodeServlet/dissertation_derivate_00007342/DissertationDanielB.pdf}
}








\bib{Claus1}{article}{
title = {Finite-dimensional hereditary algebras of wild representation type},
author = {Claus Michael Ringel},

journal = {Mathematische Zeitschrift},
volume = {161},
date = {1978},
pages = {235-255},

}


\bib{Claus3}{webpage}{
title = { Simple representations, thin representations},
author = {Michael  Ringel,Claus},

url = {https://www.math.uni-bielefeld.de/~sek/kau/leit2v2.pdf},
}


















\bib{George}{article}{
author = { George   Matsaglia  and  George   P. H. Styan },
title = {Equalities and Inequalities for Ranks of Matrices},
journal = {Linear and Multilinear Algebra},
volume = {2},
number = {3},
pages = {269-292},
year  = {1974},
publisher = {Taylor & Francis},
}


\bib{West}{article}{
  title = {Spaces of linear transformations of equal rank},
  author = {R.Westwick},
  journal = {Linear algebra and its applications},
  volume = {5},
  number = {1},
  pages = {49-64},
  year = {1972},
  publisher = {North-Holland}
}




\bib{Jie2}{webpage}{
title = { Widths of regular components for $n$-regular tree $T(n)$},
author = {Liu, Jie},

url = {https://arxiv.org/pdf/2606.06964},
}










\end{biblist}
\end{bibdiv}
\end{document}